\documentclass[a4paper,reqno]{amsart}

\usepackage{hyperref}

\usepackage{amssymb,amsxtra,amsfonts,dsfont}

\usepackage{graphicx} \usepackage{amsmath} \usepackage{amssymb}

\usepackage[svgnames]{xcolor}

\usepackage[color,matrix,arrow]{xy}

\input xy \xyoption{all}

\newtheorem{lem}{Lemma}[section]

\newtheorem{thm}[lem]{Theorem}

\newtheorem{pro}[lem]{Proposition}

\newtheorem{cor}[lem]{Corollary}

\newtheorem{exa}[lem]{Example}

\newcommand{\ttn}[1]{\tau_n^{#1}}

\newcommand{\lrw}{\longrightarrow}

\newcommand{\LL}{\Lambda}

\newcommand{\xa}{\alpha}
\newcommand{\xb}{\beta}

\newcommand{\caG}{\mathcal G}

\newcommand{\caM}{\mathcal M}

\newcommand{\RHom}{\mathbf{R}\strut\kern-.2em\operatorname{Hom}\nolimits}

\newcommand{\zZ}{{\mathbb Z}}

\newcommand{\zzs}[1]{{\mathbb Z|_{#1} }}

\newcommand{\tL}{\widetilde{\LL}}

\newcommand{\dtL}{\Delta{\LL}}

\newcommand{\olL}{\overline{\LL}}
\newcommand{\olM}{\overline{M}}

\newcommand{\GG}{\Gamma}

\newcommand{\trho}{\tilde{\rho}}

\newcommand{\olG}{\overline{\Gamma}}

\newcommand{\tQ}{\widetilde{Q}}
\newcommand{\tQs}{\widetilde{Q}^{\sigma}}

\newcommand{\olQ}{\overline{Q}}

\newcommand{\olrho}{\overline{\rho}}

\newcommand{\om}[1]{\Omega^{#1}}

\newcommand{\sk}[1]{^{(#1)}}

\newcommand{\hhom}{\mathrm{Hom}}
\newcommand{\hhm}{\mathrm{Hom}}

\newcommand{\Ext}{\mathrm{Ext}}

\newcommand{\uhom}[1]{\underline{\mathrm{Hom}}}
\newcommand{\uExt}[1]{\underline{\mathrm{Ext}}}

\newcommand{\ohom}[1]{\overline{\mathrm{Hom}}}

\newcommand{\Ker}{\mathrm{Ker}\,}

\newcommand{\cok}{\mathrm{coKer}\,}

\newcommand{\add}{\mathrm{add}\,}

\newcommand{\id}{\mathrm{id}\,}

\newcommand{\dmn}{\mathrm{dim}\,}

\newcommand{\rend}{\mathrm{End}\,^{op}}

\newcommand{\arr}[2]{\begin{array}{#1}#2\end{array}}

\newcommand{\eqqc}[2]{\begin{equation}\label{#1}#2\end{equation}}

\newcommand{\eqqcn}[2]{\[ #2 \]}

\begin{document}

\title{$n$-APR tilting and $\tau$-mutations}
\author{Jin Yun Guo and Cong Xiao}
\address{Jin Yun Guo\footnote{This work is supported by Natural Science Foundation of China \#11271119,\#11671126, and the Construct Program of the Key Discipline in Hunan Province}, Cong Xiao, LCSM( Ministry of Education), School of Mathematics and Statistics\\ Hunan Normal University\\ Changsha, Hunan 410081, P. R. China}
\email{gjy@hunnu.edu.cn,785519703@qq.com}



%\date{}
\begin{abstract}
APR tilts for path algebra $kQ$ can be realized as the mutation  of the quiver  $Q$ in $\zZ Q$ with respect to the translation.
In this paper, we show that we have similar results for the quadratic dual of truncations of $n$-translation algebras, that is,  under certain condition, the $n$-APR tilts of such algebras are realized as $\tau$-mutations.
For the dual $\tau$-slice algebras with bound quiver $Q^{\perp}$, we show that their iterated $n$-APR tilts are realized by the iterated $\tau$-mutations in $\zzs{n-1}Q^{\perp}$.

%\keywords{$n$-APR tilting \and $\tau$-mutation \and $n$-translation quiver \and $n$-translation algebra \and $\tau$-slice algebra}
%\subclass{16G20 \and 16G70\and  16S37 }
\end{abstract}

\maketitle

\section{Introduction}\label{intro}
Higher representation theory is developed by Iyama and his coauthors \cite{i007,i07,i11,io11,hio}, and is widely used in representation theory of algebra and non-commutative geometry.
We observed that graded self-injective algebra bear certain feature of higher representation theory \cite{gw00,g12}, and introduce $n$-translation algebra for studying higher representation theory related to self-injective algebras in \cite{g16}.

$n$-APR tilts are introduced by Iyama and Oppermann in \cite{io11} as generalization of Bernstein-Gelfand-Ponomarev reflection functor and the APR tilts, which played very important role in studying representation finite hereditary algebras.

For a quiver $Q$ or the path algebra $kQ$ defined by the quiver, the  Bernstein-Gelfand-Ponomarev reflection or APR tilts can be realized as follows.
Embedding $Q$ as a slice in the translation quiver $\zZ Q$ with $i$ as a sink (respectively, a source) of (the image of) $Q$. By taking 'mutation' with respect to the translation $\tau$  on $Q$, that is, by replacing $i$ with $\tau i$ (respectively, $\tau^{-1} i)$, the quiver obtained is exactly the one obtained by Bernstein-Gelfand-Ponomarev reflection or APR tilts on $Q$ at $i$.
In this paper, we show that for certain algebra related to $n$-translation quiver, the same thing holds for $n$-APR tilts and the $\tau$-mutations with respect to the $n$-translation in the $n$-translation quiver.

Let $\olQ$ be an $n$-translation quiver such that the bound algebra $\olL$ defined by it is an $n$-translation algebra (see Section \ref{pre} for the related notions).
Let $Q$ be the truncation of an $n$-translation quiver $\olQ$, and let $\LL$ be the algebra defined by $Q$ and let $\GG$ be the quadratic dual of $\LL$.
In this paper, we show that certain $n$-APR tilts for the algebra $\GG$ is obtained by applying the $\tau$-mutations on its bound quiver in $\olQ$.
In the case $Q$ is a complete $\tau$-slice of $\olQ$, $n$-APR tilts of $\GG$ are realized by the  $\tau$-mutations in $\zzs{n-1}Q^{\perp}$, hence we obtain all the iterated $n$-APR tilted algebras of $\GG$ using $\tau$-mutations.
So there are only finitely many iterated $n$-APR tilts for each $\GG$.

The paper is organized as follows.
In Section \ref{pre}, we recall concepts and results needed in this paper.
In Section \ref{sec:tilting}, we study the quadratic dual of the truncation of an  $n$-translation algebra, and characterize its tilting module and tilted algebra related to the Koszul complex given by a $\tau$-hammock.
In Section \ref{mutations}, we characterize the $n$-APR tilting module and $n$-APR tilted algebras for the forward movable sinks and backward movable sources, and prove that such $n$-APR tilts are realized by $\tau$-mutations.
Our results are applied in Section \ref{slices} for dual $\tau$-slice algebras, we show that for dual $\tau$-slice algebra with bound quiver $Q^{\perp}$, its iterated $n$-APR tilts are realized by the iterated $\tau$-mutations in $\zzs{n-1}Q^{\perp}$.
We also present two examples in this section using $\tau$-mutations.
In one example, we recover the list of iterated $2$-APR tilts in \cite{io11} for $2$-representation-finite Auslander algebra of type $A_3$, and in the other example, we give a list of iterated $2$-APR tilts of quasi $1$-Fano algebras related to the McKay quiver of type $D_4$ in $\mathrm{SL}(\mathbb C^3)$.

This paper is an extended and generalized version of the results concerning $n$-APR tilts in '$\tau$-slice algebras of $n$-translation algebras and quasi $n$-Fano algebras, arXiv:1707.01393', which is discontinued.

\section{Preliminary}\label{pre}

Let $k$ be a field, and let $\LL = \LL_0 + \LL_1+\cdots$ be a graded algebra over $k$ with $\LL_0$ direct sum of copies of $k$ such that $\LL$ is generated by $\LL_0$ and $\LL_1$.
Such algebra is determined by a bound quiver $Q= (Q_0,Q_1, \rho)$ \cite{g16}. A module means a left module in this paper, when not specialized.

Recall that a bound quiver $Q= (Q_0, Q_1, \rho)$ is a quiver with $Q_0$ the set of vertices, $Q_1$ the set of arrows and $\rho$ a set of relations.
In this paper the vertex set $Q_0$ may be infinite.
The arrow set $Q_1$ is as usual defined with two maps $s, t$ from $Q_1$ to $Q_0$ to assign an arrow $\alpha$ its starting vertex $s(\alpha)$ and its ending vertex $t(\alpha)$.
The arrow set $Q_1$ in this paper is assumed to be locally finite in the sense that for each $i \in Q_0$, the number of arrows $\alpha$ with $s(\alpha)=i$ or $t(\alpha)=i$ is finite.
We also write $s(p)=i$, $t(p)=j$ for a path in $Q$ from $i$ to $j$.
The relation set $\rho$ is a set of linear combinations of paths of length $\ge 2$, we may assume that it is {\em normalized} in the sense that each element in $\rho$ is a linear combination of paths starting at the same vertex and ending at the same vertex.
Since we study graded algebra, we may assume that  $\rho$ is {\em homogeneous}, that is, the paths appearing in each elements in $\rho$ are of the same length.

Let $\LL_0 = \bigoplus\limits_{i\in Q_0} k_i =kQ_0$, with $k_i \simeq k$ as algebras, and let $e_i$ be the image of the identity of $k$ under the canonical embedding of the $k_i$ into $kQ_0$.
Then $\{e_i| i \in Q_0\}$ is a complete set of orthogonal primitive idempotents in $kQ_0$ and $kQ_1= \LL_1 = \bigoplus\limits_{i,j \in Q_0 }e_j \LL_1 e_i$ as $kQ_0 $-bimodules.
Note that $\LL$ is nounital when $Q_0$ is infinite.
Fix a basis $Q_1^{ij}$ of $e_j \LL_1 e_i$ for any pair $i, j\in Q_0$, take the elements of $Q_1^{ij}$ as arrows from $i$ to $j$, and let $Q_1= \cup_{(i,j)\in Q_0\times Q_0} Q_1^{ ij}.$
Let $Q_t$ be the set of the paths of length $t$ in $Q$ and let $kQ_t$ be the $k$-space spanned by $Q_t$.

There is canonical epimorphism from $kQ$ to $\LL$ with kernel $I$ contained in $kQ_2 + kQ_3 + \cdots$.
Choose a generating set $\rho$  of $I$, whose elements are linear combinations of paths of length $\ge 2$.
Then we have $I=(\rho)$ and $\LL \simeq kQ/(\rho)$, a quotient algebra of the path algebra $k Q$.
$\LL$ is also called the bound quiver algebra of the bound quiver $Q= (Q_0,Q_1, \rho)$.
A path $p$ in $Q$ is called bound path if its image in $kQ/(\rho)$ is non-zero.

A quiver $Q$ is called {\em acyclic} if $Q$ contains no oriented cycle, the algebra $\LL$ of the bound quiver $Q$ is called {\em acyclic} if its quiver $Q$ is acyclic.

Let $Q= (Q_0, Q_1, \rho)$ be a bound quiver with quadratics relations, that is, $\rho$ is a set of linear combination of paths of length $2$.
In this case, the quotient algebra $\LL = kQ/(\rho)$ is called a {\em quadratic algebra.}
We identify $Q_0$ and $Q_1$ with their dual bases in the local dual spaces $\bigoplus_{i \in Q_0} D k e_i$ and $\bigoplus_{i,j \in Q_0} De_j kQ_1 e_i$, that is, define
\eqqcn{dualdefa}{e_i(e_j) =  \left\{ \arr{ll}{1, & i=j\\0,& i\neq j} \right. \mbox{ and } \xa(\xb)  = \left\{\arr{ll}{1, & \xa=\xb\\0,& \xa\neq \xb}\right. , } for $i,j \in Q_0, \xa,\xb\in Q_1$  and set \eqqc{dualdef}{ \xb_t \cdots \xb_1 (\xa_t  \cdots  \xa_1)= \xb_t(\xa_t)\cdots\xb_1(\xa_1),} for  paths $\xb_t\cdots\xb_1,\xa_t \cdots \xa_1$ in $e_j kQ_t e_i$.
Then paths in $e_j kQ_t e_i$ are identified with the corresponding elements in dual basis in $D e_j kQ_t e_i$ for each pair $i,j$ of vertices and each $t\ge 1$.
Take a spanning set $ \rho^{\perp}_{i,j}$ of orthogonal subspace of the set $e_j k\rho e_i$ in the space $D e_j kQ_2 e_i = e_j kQ_2 e_i$ for each pair $i,j \in Q_0$, and set \eqqcn{relationqd}{\rho^{\perp} = \cup_{i,j\in Q_0} \rho^{\perp}_{i,j},}
The algebra $\LL^{!,op} = k Q^{\perp} /(\rho^{\perp})$ is called the {\em quadratic dual} of $\LL$.

\medskip

In \cite{g16}, we also introduce $n$-translation quivers and $n$-translation algebras.
Recall that a bound quiver $Q=(Q_0,Q_1,\rho)$, with $\rho$ homogeneous relations, is called an {\em $n$-translation quiver} if there is a bijective map $\tau: Q_0 \setminus \mathcal P \longrightarrow Q_0 \setminus \mathcal I$, called the {\em $n$-translation} of $Q$, for two subsets $\mathcal P$ and $ \mathcal I $ of $Q_0$, whose elements are called {\em projective vertices} and respectively {\em injective vertices},  satisfying the following conditions for the quotient algebra $\LL=kQ/(\rho)$:

1. Any maximal bound path is of length $n+1$ from  $\tau i$ in $Q_0 \setminus \mathcal I$ to $i$, for some vertex $i$ in $Q_0 \setminus \mathcal P $.

2. Two bound paths of length $n+1$ from $\tau i$ to $i$ are linearly dependent, for any $i \in Q_0 \setminus \mathcal P$.

3. For each $i\in Q_0\setminus \mathcal P$ $\LL e_i$ is a projective injective module of Loewy length $n+2$ and and each $j\in Q_0$ such that there is a bound path from $j$ to $i$ of length $t$ or there is a bound path from $\tau i$ to $j$, then the multiplication of the quotient algebra $\LL=kQ/(\rho)$ defines a non-degenerated bilinear from $e_i\LL_t e_{j} \times e_j\LL_{n+1-t} e_{\tau i} $ to $e_i\LL_{n+1} e_{\tau i} $, here $\LL_t$ is the subspaces spanned by the images of the paths of length $t$.

Taking a connected component of the Auslander-Reiten quiver of an algebra with the mesh relations, one get a quadratic bound quiver.
Its quadratic dual quiver is a $1$-translation quiver, with the Auslander-Reiten translation as its $1$-translation (see \cite{g02}).

$Q$ is called {\em stable} if $\mathcal I=\mathcal P =\emptyset$, and this is the case if and only if $\LL$ is a self-injective algebra.

An algebra $\LL$ with an $n$-translation quiver $Q=(Q_0,Q_1, \rho)$ as its bound quiver is called an {\em $n$-translation algebra}, if there is a $q \in \mathbb N \cup \{\infty \}$ such that $\LL$ is $(n+1,q)$-Koszul.
An $n$-translation algebra is quadratic and its quadratic dual is $(q,n+1)$-Koszul \cite{g16,bbk02}.
\medskip

$\tau$-hammocks are introduced in \cite{g12} for stable $n$-translation quiver, and are extended to $n$-translation quiver in \cite{g16}, which are generalization of meshes in a translation quiver.
Let $\olQ= (\olQ_0,\olQ_1, \overline{\rho} )$ be an $n$-translation quiver with $n$-translation $\tau$ and let $\olL = k\olQ/(\olrho)$.
For a non-injective vertex  $i\in \olQ_0$,  $\tau^{-1} i $ is defined in $\olQ$,  the {\em $\tau$-hammock $H_i=  H_{i, \olQ}$ ending at $i$} is defined  as the quiver with the vertex set $$H_{i,0}=\{ (j,t) | j \in \olQ_0, \exists p\in \olQ_t, s(p)=j, t(p)=i,  0\neq p\in \olL \},$$ the arrow set $$\arr{ll}{ H_{i,1} = &\{ (\alpha , t): (j,t+1) \longrightarrow (j', t)|\alpha: j\to j'\in \olQ_1, \\ &\qquad \exists p\in \olQ_t, s(p)=j', t(p)=i, 0\neq p\xa \in \olL_{t+1} \},} $$ and a hammock function $\mu_i: H_{i,0} \longrightarrow \zZ$ which is integral map on the vertices defined by $$\mu_i (j,t)=\dmn_k e_i \olL_t e_j$$ for $(j,t)\in H_{i,0}$.

Similarly, for a non-projective vertex $i$,  $\tau i$ is defined in $\olQ$, the  {\em $\tau$-hammock $H^i= H^i_{\olQ}$ starting at $i$} is defined as the quivers with the vertex set $$H^i_0=\{ (j,t) | j \in \olQ_0, \exists p\in \olQ_t,  s(p)=i, t(p)=j, 0\neq p\in \olL \}, $$ the arrow set
$$\arr{ll}{ H^i_1=\{ (\alpha , t): (j,t) \longrightarrow (j', t+1)|\alpha: j\to j'\in \olQ_1, \\ \qquad \exists p\in \olQ_t, s(p)=i, t(p)=j, 0\neq \alpha p\in \olL_{t+1}  \},}$$ and a hammock functions $\mu^i: H^i_0 \longrightarrow \zZ$ which is the integral map on the vertices  defined by $$\mu^i (j,t)=\dmn_k e_j \olL_t e_i$$ for $(j,t)\in H^i_0$.

We also denote by $H^i_0$ and $H_{i,0}$ the sets $\{j \in \olQ_0| (j,t) \in H^i_0 \mbox{ for some } t\in \zZ\}$ and $\{j \in \olQ_0| (j,t) \in H_{i,0} \mbox{ for some } t\in \zZ\}$, respectively.
When $\olQ$ is acyclic, the hammocks $H^i$ and $H_i$ is regarded as a sub-quiver of $\olQ$ via projection to the first component of the vertices and arrows.
In this case, $H_i = H^{\tau^{-1} i}$ when $\tau^{-1} i$ is defined and $H^{i} = H_{\tau i }$ when $\tau i$  is defined.

\medskip

$\tau$-hammocks describe the indecomposable projective injective modules of $n$-translation algebras.
Using Koszul duality, they also describe the Koszul complexes of its quadratic dual.

Let $\olL$ be an $n$-translation algebra with bound quiver $\olQ =(\olQ_0, \olQ_1, \overline{\rho} )$, and let $\olG$ be the quadratic dual of $\olL$ with bound quiver $\olQ^{\perp} =(\olQ_0, \olQ_1, \overline{\rho}^{\perp} )$.
Let $\caG = \add(\olG)$ be the category of finite generated projective $\olG$-modules.
The Koszul complexes in $\olG$  are the $n$-almost split sequences in $\caG$ under certain conditions \cite{g16}.
If $i$ is a non-injective vertex in $\olQ$, (or $j=\tau^{-1} i$ is a non-projective vertex), by \cite{bgs96}, we have a left Koszul complex,
\small\eqqc{Koszulcomplex}{\arr{l}{
\olG \otimes \olL_{n+1} e_{\tau^{-1} i} \longrightarrow\olG\otimes \olL_n e_{\tau^{-1}i} \stackrel{\psi}{\longrightarrow} \olG \otimes \olL_{n-1} e_{\tau^{-1} i}\longrightarrow \cdots \\ \qquad \qquad \qquad\longrightarrow\olG\otimes \olL_2 e_{\tau^{-1}i} \stackrel{\phi}{\longrightarrow} \olG \otimes \olL_{1} e_{\tau^{-1} i} \longrightarrow \olG \otimes \olL_{0} e_{\tau^{-1} i}
}}\normalsize
By the definition of the hammocks, we get the following characterization of  the Koszul complexes in $\olG$.

\begin{pro}\label{KoszulinG} For each vertex $i\in \olQ_0$ such that $\tau^{-1} i$ is defined in $\olQ$, we have a Koszul complex
\eqqcn{nassinsubG}{\xi_i: M_{n+1}=\olG e_{ i } \lrw \cdots \lrw M_t \lrw \cdots \lrw M_0= \olG e_{\tau^{-1}i}
} with $M_t = \bigoplus\limits_{(j,n+1-t)\in H_{i,0}}  (\olG e_j)^{\mu_i(j,n+1-t)} $ for $0\le t\le n$ and for each vertex $i\in \olQ_0$ such that $\tau i$ is defined in $\olQ$, a Koszul complex
\eqqcn{nassinsubG1}{\zeta_i: M_{n+1}= \olG e_{\tau i } \lrw \cdots \lrw M_t \lrw \cdots \lrw M_0=\olG e_{i}
} with $M_t = \bigoplus\limits_{(j,t)\in H^i_0 } (\olG e_j)^{\mu^i(j,t)} $ for $0\le t\le n$.
\end{pro}

We have that $\zeta_i =\xi_{\tau i}$ and $\xi_i =\zeta_{\tau^{-1}i}$.
We also call \eqref{Koszulcomplex} the Koszul complex related to the $\tau$-hammock $H_i=H^{\tau^{-1} i}$.

Under certain conditions, these Koszul complexes are $n$-almost split sequences in the category of its finite generated projective modules of $\olG$.
\section{Tilting for the Truncation of $n$-translation algebras}\label{sec:tilting}

We study the tilting  for the algebra defined by the dual of a truncation of an $n$-translation algebra.
We first introduce the truncation of a stable $n$-translation algebra.

Let $\olQ= (\olQ_0,\olQ_1, \overline{\rho} )$ be a bound  quiver, and let $\olL\simeq k\olQ/(\overline{\rho})$.
Let $Q=(Q_0,Q_1)$ be a finite full sub-quiver of $\olQ$, that is, $Q_0$ and $Q_1$ are subsets of $\olQ_0$ and $\olQ_1$, respectively, such that all the arrows of $\olQ$ from $i$ to $j$ are in $Q_1$ whenever $i$ and $j$ are both in $Q_0$.
A path $p = \xa_h\cdots\xa_1$ from $i$ to $j$ in $\olQ$ is said to be in a sub-quiver $Q=(Q_0,Q_1)$ if all the vertices $s(\xa_1)=i, t(\xa_1)=s(\xa_2) \ldots,t(\xa_{h-1})=s(\xa_h), t(\xa_h)=j$ are in $Q_0$ and all the arrows $\xa_1,\cdots,\xa_h$ are in $Q_1$.
A full bound sub-quiver $Q$ of $\olQ$ and we say it is a {\em convex truncation} of $\olQ$ if there is a path from $i$ to $j$ is in $Q$, then any path from $i$ to $j$ is also in $Q$.
Since $\overline{\rho}$ is normalized, $\rho=\{e_j x e_i| x\in \olrho, i,j\in Q_0\}$ is a subset of $\overline{\rho}$.
Thus a  convex truncation $Q = (Q_0,Q_1,\rho)$ is a bound sub-quiver of $\olQ$.

If $\olL$ is quadratic, write $\olG = \olL^{!.op}$ for its quadratic dual.
For a full subquiver $Q=(Q_0,Q_1)$  of $\olQ$, we have two algebras associate to it, that is, the subalgebra $$\LL(Q) = e\olL e$$ with $e =\sum\limits_{j\in Q_0}e_j$ and the quotient algebra $$\LL'(Q) = \olL / (\{e_j|j\in \olQ_0\setminus Q_0\}).$$
Since $Q$ is finite, $\LL(Q)$ is an unital algebra with the unit $e$.

Write $$L(Q) =\bigoplus_{i\in Q_0} \olG e_i \mbox{ and }L\sk{j}(Q) =\bigoplus_{i\in Q_0\setminus \{j\}} \olG e_i.$$
For an elements $x = \sum_{p \in \cup_{t \ge 2}\olQ_t} a_p p$ in $k\olQ$, write $x_Q = \sum_{p \in \cup_{t \ge 2}Q_t} a_p p$ and let $\rho = \{ x_Q | x\in \olrho\}$, then we have $$\LL'(Q) \simeq kQ/(\rho).$$
If $\olQ$ is quadratic, $\LL'(Q)$ is quadratic and write $$\GG'(Q) =(\LL'(Q))^{!,op} ,$$ the quadratic dual of $\LL'(Q)$.
We have the following proposition.

\begin{pro}\label{truncation-alg} Assume that  $Q=(Q_0,Q_1,\rho)$ is a convex truncation of a bound quiver $\olQ= (\olQ_0,\olQ_1, \overline{\rho} )$.
Then  $\LL(Q) \simeq \LL'(Q)$ is both a subalgebra and a quotient algebra of $\olL$.
\end{pro}
\begin{proof}
Clearly,  $$\LL(Q)=e \olL e \simeq (\hhm_{\olL}(\olL e, \olL e))^{op} = \rend_{\olL} \bigoplus\limits_{j \in Q_0} \olL e_j$$ is a subalgebra of $\olL$ generated by $\{e_j |j \in Q_0\}$, and $e\olQ_1 e \setminus\{0\} = Q_1$.
Thus $\LL \simeq kQ/J$ for some ideal $J$ of $kQ$.
Since $J$ is a set of elements in $kQ$ whose image in $\olL$ is zero, it is in the ideal $(\overline{\rho})$, so $$J\subseteq kQ\cap (\overline{\rho}) = kQ \cap(\rho) =(\rho).$$
If $z \in (\rho) = kQ\cap(\overline{\rho})$, then the image of $z$ in $\olL$ is zero, hence its image in $\LL(Q)$ is zero, too.
This shows that $J =(\rho)$, and $$\LL(Q)= e\olL e\simeq  kQ/(\rho) \simeq \LL'(Q).$$
\end{proof}

We call the algebra $\LL$ a {\em truncation} of $\olL$ if $Q$ is a convex truncation.
If $\olL$ is quadratic,  so is $\LL$, and we call its quadratic dual $\GG=\GG(Q) =\LL^{!,op}$ a {\em dual truncation} of $\olL$.

\medskip

Assume that $\olL$ is an acyclic $n$-translation algebra with bound quiver $\olQ= (\olQ_0, \olQ_1, \olrho )$.
Recall that if $\olL$ is $(n+1,q)$-Koszul with finite $q$, then its quadratic dual $\olG$ is $(q,n+1)$-Koszul, by Proposition 3.1 of \cite{bbk02}, and $\olG$ is a $(q-1)$-translation algebra \cite{g16}, write its $(q-1)$-translation as $\tau_{\perp}$.
If $\olL$ is Koszul, this include the case $q=\infty$, then $\tau_\perp $ is not defined, and we conventionally assume that $\tau_{\perp}^{t}i \not \in \olQ_0$ for any $i\in \olQ_0$ and $t\in \zZ$ in this case.

A non-injective vertex $i$ in $Q_0$ is called {\em  forward movable} provided that $H_{\olL,i,0} \setminus Q_0 = \{\tau^{-1} i\}$ and $\tau_\perp i\not \in Q_0$, and a non-projective vertex $i$ in $Q_0$ is called {\em  backward movable} provided that  $H^i_{\olL,0} \setminus Q_0 = \{\tau i\}$ and $\tau_\perp^{-1} i\not \in Q_0$.
If $i$ is a forward movable vertex of $Q$, let $s^-_i Q$ be the full subquiver of $\olQ$ with the vertex set $(Q_0\setminus \{i\})\cup \{\tau^{-1} i\}$, and if $i$ is a backward movable vertex of $Q$, let $s^+_i Q$ be the full subquiver of $\olQ$ with the vertex set $(Q_0\setminus \{i\})\cup \{\tau i\}$.

Let $\caG = \add(\olG)$ be the category of finite generated projective $\olG$-modules.

If $Q=(Q_0,Q_1,\rho)$ is a finite truncation of $\olQ$, set $$L = L_Q = \bigoplus\limits_{j\in Q_0} \olG e_j , \mbox{ and }L\sk{i}=\bigoplus\limits_{j\in Q_0\setminus\{i\}} \olG e_j .$$
Consider the Koszul complex \eqref{Koszulcomplex} of $\add \olG$ starting at $i$, we have the following lemma.
\begin{lem}\label{kernelKcomplex}
If $i$ is a forward movable vertex of $Q$, then
$$
\cok \hhm_{\caG} (L, \phi) \simeq \hhm_{\caG } (L,\olG \otimes \olL_{0} e_{\tau^{-1} i}) = \hhm_{\caG } (L,\olG  e_{\tau^{-1} i}). $$

If $\tau^{-1} i$ is backward movable vertex of $Q$, then
$${\cok \hhm_{\caG} ( \psi, L) \simeq \hhm_{\caG } (e_i \olG \otimes \olL_{0},  L) = \hhm_{\caG } (e_{i}\olG, L).
}$$
\end{lem}
\begin{proof} Note that $\hhom_{\caG}(\olG e_j, \olG e_i)/J_{\caG}(\olG e_j, \olG e_i) =0$ for $i\neq j$ since by our assumption, $\tau^{-1}_{\perp} i \not \in Q_0$ and the Koszul complex \eqref{Koszulcomplex} is an $n$-almost split sequence in $\caG$, by Theorem 7.2 of \cite{g16}.
So $\hhm_{\olG }(L, \olG_0 e_{\tau^{-1} i})/J_{\olG }(L, \olG_0 e_{\tau^{-1} i})=0 $, and we have exact sequence
$$\arr{l}{\hhm_{\caG } (L,\olG\otimes \olL_2 e_{\tau^{-1}i}) \stackrel{\hhm_{\caG } (L,\phi)}{\longrightarrow} \hhm_{\caG } (L, \olG \otimes \olL_{1} e_{\tau^{-1} i}) \\ \qquad \qquad
\longrightarrow  \hhm_{\caG } (L,\olG \otimes \olL_{0} e_{\tau^{-1} i}) \longrightarrow  0.
}$$
And the first equality follows.

The second equality is proven  dually.
\end{proof}

Recall that a module $T$ over an algebra $\GG$ is called a {\em tilting module} if

(1). $T$ has a finite projective resolution $$0 \lrw P_n \lrw \cdots \lrw P_0 \lrw T \lrw 0$$ with each $P_t$  finitely generated projective $\GG$-module;

(2). $\Ext_{\GG}^t(T,T) = 0$ for all $t > 0$;

(3). There is an exact sequence $$0 \lrw \GG \lrw T_0 \lrw \cdots \lrw T_m \lrw 0$$ of $\GG$-modules with each $T_t$ in $\add(T)$.

\medskip

A cotilting module is defined dually.

\medskip

Now we study the tilts and cotilts on $\GG$ related to the Koszul complex \eqref{Koszulcomplex}.
We first prove the following theorem.
\begin{thm}\label{tilting}
Let $\olL$ be an acyclic $n$-translation algebra with $n$-translation quiver $\olQ$ and $n$-translation $\tau$ and let $\olG \simeq k\olQ/(\olrho^{\perp})$ be its quadratic dual.
Let $Q = (Q_0,Q_1,\rho)$ be a convex truncation of $\olQ$, $L = L(Q)$ and  $\GG = \GG(Q)$.

(1). If  $i$ is a forward movable vertex in $Q$, set $$T= \hhm_{\caG} (L,L\sk{i}(Q)) \oplus \cok \hhm_{\caG} (L, \phi),$$ where $\phi: \olG\otimes \olL_2 e_{\tau^{-1}i}\to \olG\otimes \olL_1 e_{\tau^{-1}i}$ is the map in the Koszul complex \eqref{Koszulcomplex}.
Then $T$ is a tilting $\GG$-module of projective  dimension at most $n$, and  $$\rend_{\GG} T \simeq \rend_{\caG} L(s^-_i Q).$$

(2). If $\tau^{-1} i$ is a backward movable vertex  in $Q$, set $$T'= \hhm_{\caG} (L\sk{\tau^{-1}i}(Q) , L) \oplus \cok \hhm_{\caG} (\psi,L),$$ where $\psi: \olG\otimes \olL_{n} e_{\tau^{-1}i}\to \olG\otimes \olL_{n-1} e_{\tau^{-1}i}$ is the map in the Koszul complex \eqref{Koszulcomplex}.
Then $T'$  is a cotilting $\GG$-module of injective  dimension at most $n$, and  $$\rend_{\GG} T' \simeq \rend_{\caG} L(s^+_{\tau^{-1}i} Q).$$
\end{thm}
We call the tilting module $T$ in (1) (respectively, $T'$ in (2)) of Theorem \ref{tilting} the {\em tilting (respectively,  cotilting)  $\GG$-module related to Koszul complex \eqref{Koszulcomplex}} for vertex $i$ (respectively, for $\tau^{-1} i$).

\begin{proof}
We prove the first case, the second case is proven dually.

Assume  that $i$ is forward movable in $Q$ and let $H_i=H_{\olL,i}$ be the hammock ending at $i$ with hammock function $\mu_i$.
\eqref{Koszulcomplex} is a truncation of the projective resolution of the simple $\olG$-module $\olG_0 e_{\tau^{-1}} i$.

By Proposition \ref{KoszulinG}, $$\olG \otimes \olL_t e_{\tau^{-1} i} \simeq \bigoplus\limits_{(j,n+1-t)\in H_{i,0}}(\olG_t e_j)^{\mu_i(j,n+1-t)} .$$
Let $M = L\sk{i}(Q)$ and let $$M_t = \bigoplus\limits_{(j,n+1-t)\in H_{i,0}} (\olG e_j)^{\mu^i(j,n+1-t)}.$$
Let $X= \olG e_{i}$ and let $Y= \olG e_{\tau^{-1} i}$,  the Koszul complex \eqref{Koszulcomplex} becomes
\eqqc{seq1}{X \stackrel{f}{\to} M_1 \to \ldots \to M_n \stackrel{g}{\to} Y .}
By Lemma 7.1 of \cite{g16}, $f$ is  a left $\add(M)$-approximation and $g: M_n \to Y$ is a right $\add(M)$-approximation.

Let $V= X\oplus M$ and $W=Y \oplus M$.
If $\olL$ is Koszul, so is $\olG$ and $\Ker f=0$, otherwise, $\olG$ is almost Koszul and $\Ker f \simeq \olG_0 e_{\tau_{\perp} i}$.
Thus $\hhm_{\caG}(V, \Ker f) =0$ since $i$ is forward  movable in $Q$.
So we get an exact sequence
\small$$
0 \lrw \hhm_{\caG}(V, X) \lrw \hhm_{\caG}(V, M_1) \lrw \ldots \lrw \hhm_{\caG}(V, M_n) \lrw \hhm_{\caG}(V, Y).$$\normalsize
Note that $\cok g \simeq \olG_0 e_{\tau^{-1} i}$.
Since $Q$ is convex and ${\tau^{-1} i}\not \in Q_0$ and $({\tau^{-1} i})^{-} \subseteq Q_0$, thus $\hhm_{\caG}(\cok g, W) =0$.
So we have an exact sequence
\small$$0 \lrw \hhm_{\caG}(Y, W) \lrw \hhm_{\caG}(M_n,W) \lrw \ldots \lrw \hhm_{\caG}(M_1, W) \lrw \hhm_{\caG}(X, W).
$$\normalsize

Note  $T= \hhm_{\GG} (L, L\sk{i}) \oplus \cok \hhm_{\caG} (L, \phi)$, where $\phi: \olG\otimes \olL_2 e_{\tau^{-1}i}\to \olG\otimes \olL_1 e_{\tau^{-1}i}$ is the map in the Koszul complex.
So by Proposition 3.4 of \cite{hx11}, $T$ is a tilting $\GG$-module of projective  dimension at most $n$.

By Lemma \ref{kernelKcomplex}, $$  \cok \hhm_{\caG} (L, \phi) \simeq \hhm_{\caG } (L,\olG  e_{\tau^{-1} i}), $$ thus
$$\arr{lll}{
\rend_{\GG} T &=& \hhm_{\GG}(\hhm_{\caG} (L, L\sk{i}(Q)) \oplus \cok \hhm_{\caG} (L, \phi), \\
&& \qquad\hhm_{\caG} (L, L\sk{i}(Q)) \oplus \cok \hhm_{\caG} (L, \phi))\\
&=&\hhm_{\GG}( \hhm_{\caG} (L, L\sk{i}(Q)),\hhm_{\caG} (L, L\sk{i}(Q))) \\&&  \oplus \hhm_{\GG}(\hhm_{\caG} (L, L\sk{i}(Q)), \cok \hhm_{\caG} (L, \phi)) \\&& \oplus \hhm_{\GG}( \cok \hhm_{\caG} (L, \phi),\hhm_{\caG} ( L,L\sk{i})) \\&& \oplus \hhm_{\GG} ( \cok \hhm_{\caG} (L, \phi),\cok \hhm_{\caG} (L, \phi))\\
&=&\hhm_{\caG}(L\sk{i}, L\sk{i}(Q))   \oplus \hhm_{\caG}( L\sk{i}(Q), \olG e_{\tau^{-1} i}) \\&& \oplus \hhm_{\GG}(\olG e_{\tau^{-1} i}, L\sk{i}(Q)) \oplus \hhm_{\caG} ( \olG  e_{\tau^{-1} i},\olG  e_{\tau^{-1} i})\\
&\simeq&\hhm_{\caG}(\bigoplus\limits_{j \in (s^-_iQ)_0} \olG e_j, \bigoplus\limits_{j \in (s^-_iQ)_0} \olG e_j) = \rend_{\caG }(\bigoplus\limits_{j \in (s^-_iQ)_0} \olG e_j)\\ && = \rend_{\caG }L(s^-_iQ) .
}$$
This proves the theorem.
\end{proof}

We remark that for a convex $Q$ in $\olQ$, $s^-_i Q$ (respectively, $s^+_i Q$) may not be convex in general, and $s^-_i Q$ (respectively, $s^+_i Q$) is convex if $i$ is a sink (respectively, source).

\section{$n$-APR tilts and $\tau$-mutations}\label{mutations}

Let $\GG$ be a finite dimensional algebra, recall that $n$-Auslander-Reiten translations are defined by $\ttn{} = D\, Tr \om{n-1}$ and $\ttn{-1} =Tr\, D \om{-(n-1)}$ \cite{i11}.
Let  $P$ be a simple projective $\GG$-module satisfying $\GG = P \oplus Q$.
$\ttn{-1} P\oplus Q$ is called an {\em $n$-APR tilting module} associated to $P$, if  $\id P=n $ and  $\Ext_{\GG}^t(D\GG, \GG e_i) =0$ for $0\le t < n$.
An {\em $n$-APR cotilting module} is defined dually \cite{io11}.

Throughout this section we assume that $\olL$ is an acyclic $n$-translation algebra with $n$-translation quiver $\olQ$ and $n$-translation $\tau$.
Let $\olG \simeq k\olQ/(\olrho^{\perp})$ be its quadratic dual.
Let $Q = (Q_0,Q_1,\rho)$ be a convex truncation of $\olQ$, and let $\GG=\GG(Q)$.
We have the following Proposition for the forward movable sinks and backward movable sources.

\begin{pro}\label{n-apr}
If $i$ is a  forwards movable sink in $Q_0$,  then $$\ttn{-1} \GG e_i  \oplus \GG (1-e_i) $$ is an $n$-APR tilting module of $\GG$.

If $\tau^{-1} i$ is a  backwards movable source in $Q_0$,  then $$\ttn{} D e_{\tau^{-1}i} \GG \oplus D (1-e_i) \GG$$ is an $n$-APR cotilting module of $\GG$.
\end{pro}
\begin{proof}
We prove the first assertion, the second follows dually.

Note that $\olL$ is also right $n$-translation algebra with the $n$-translation $\tau^{op} = \tau^{-1}$.
We have a right Koszul complex
$$ e_{\tau^{-1} i}\olG \lrw \olM_1 \lrw \cdots \lrw \olM_t \lrw \cdots \lrw \olM_{n} \lrw e_{i}\olG  $$
which is the projective resolution of the right simple $e_i \olG_0$, and we have $$\olM_t =\bigoplus\limits_{(j,n+1-t) \in H_{i,0}} e_j\olG^{\mu_i(j,n+1-t)} .$$
This induces a complex of right projective $\GG$-modules which is a projective resolution of the right simple $\GG$ module $S_i = e_i\GG_0 $.
$$e_{\tau^{op,-1} i}\olG e \lrw M_{1} \lrw \cdots \lrw M_t \lrw \cdots \lrw M_n \lrw e_{i}\GG,  $$
and $$M_t =\bigoplus\limits_{(j,n+1-t) \in H_{i,0}} e_j\GG^{\mu_i(j,n+1-t)} .$$
Apply the duality  $D= \hhm_k(\quad, k)$, one gets an injective resolution of left simple $\GG$-module $S_i=\GG_0 e_i$,
\eqqc{injresof}{D e_{i}\GG\lrw D M_n\lrw \cdots \lrw D M_t\lrw\cdots\lrw D M_{1} \lrw D e_{\tau^{op,-1} i}\olG e.}
Apply $\hhom_{\GG}(D  \GG, -)$, one gets
$$\arr{l}{0 \lrw\hhom_{\GG}(D \GG, S_i) \lrw \hhom_{\GG}(D  \GG, De_i \GG) \lrw \cdots \lrw \hhom_{\GG}(D \GG, D M_t)\\ \qquad\lrw \cdots \lrw \hhom_{\GG}(D  \GG, D M_{1})\lrw \hhom_{\GG}(D  \GG, D e_{\tau^{op,-1} i}\olG e)} $$
Write $_{\olG} (-,-) $ for $\hhom_{\olG} (-,-) $ and $_{\GG} (-,-) $ for $\hhom_{\GG} (-,-) $, we have the following commutative diagram with isomorphisms between rows:
\small
$$
\arr{ccccccc}{
 _{\GG}(D \GG, D e_i \GG ) &\lrw \cdots \lrw & _{\GG}(D  \GG, \bigoplus\limits_{j} D e_j\GG^{\mu_i(j,t)}) &\lrw \cdots \lrw &  _{\GG}(D  \GG, \bigoplus\limits_{j} e_j \GG^{\mu_i(j,1)})\\
\downarrow \simeq &&\downarrow \simeq &&\downarrow \simeq\\
_{\GG}(e_i \GG,   \GG )&\lrw \cdots \lrw & \bigoplus\limits_{j} {}_{\GG}( e_j\GG^{\mu_i(j,t)},  \GG
) &\lrw \cdots \lrw & \bigoplus\limits_{j} {}_{\GG}(e_j \GG^{\mu_i(j,1)} ,  \GG )\\
\downarrow \simeq &&\downarrow \simeq &&\downarrow \simeq\\
\GG e_i &\lrw \cdots \lrw& \bigoplus\limits_{j}( \GG e_j)^{\mu_i(j,t)} &\lrw \cdots \lrw &  \bigoplus\limits_{j}(\GG e_j)^{\mu_i(j,1)} \\
\downarrow \simeq &&\downarrow \simeq &&\downarrow \simeq\\
e \olG e_i &\lrw \cdots \lrw& \bigoplus\limits_{j} (e\olG e_j)^{\mu_i(j,t)} &\lrw \cdots \lrw & (e \olG e_{j})^{\mu_i(j,1)})\\
\downarrow \simeq &&\downarrow \simeq &&\downarrow \simeq\\
_{\olG}(\olG e,   \olG e_i )&\lrw \cdots \lrw & \bigoplus\limits_{j} {}_{\olG}(\olG e,  (\olG e_j)^{\mu_i(j,t)}) &\lrw \cdots \lrw & \bigoplus\limits_{j} {}_{\olG}( \olG e, ( \olG e_j)^{\mu_i(j,1)}).\\
}$$
\normalsize
Since $\tau^{op,-1} i = \tau i$ is not a vertex in $Q$, $\olQ$ is acyclic and $Q$ is convex in $\olQ$, none of the vertex prior $\tau i$ is in $Q$.
So $ {\hhom}_{\olG}( \olG e,  \olG e_{\tau{} i})=0$ and we have that the lowest row is an exact sequence $$\arr{l}{0\to \hhom_{\olG}(\olG e,   \olG e_i )\lrw \cdots \lrw  \bigoplus\limits_{j} {\hhom}_{\olG}(\olG e,  (\olG e_j)^{\mu_i(j,t)}) \lrw \cdots \\ \qquad\lrw  \bigoplus\limits_{j} {\hhom}_{\olG}( \olG e, ( \olG e_j)^{\mu_i(j,1)}) \lrw {\hhom}_{\olG}( \olG e,  \olG e_{\tau{} i})=0. }$$
So $$\arr{l}{ 0\to  \hhom_{\GG}(D \GG, D e_i \GG ) \lrw \cdots \lrw  \hhom_{\GG}(D  \GG, \bigoplus\limits_{j} D e_j\GG^{\mu_i(j,t)}) \\ \qquad\lrw \cdots \lrw   \hhom_{\GG}(D  \GG, \bigoplus\limits_{j} e_j \GG^{\mu_i(j,1)}) \lrw 0}
$$
is exact and thus $\Ext_{\GG}^t(D\GG, \GG e_i) =0$ for $0\le t < n$.
On the other hand, $$e_{\tau^{op,-1} i}\olG e = {\hhom}_{\olG}( \olG e,  \olG e_{\tau{} i})=0,$$  so $\id S_i =n$ by \eqref{injresof}.
This proves that  $\ttn{-1} \GG e_i  \oplus \GG (1-e_i) $ is an $n$-APR tilting module of $\GG$.
\end{proof}

Let $Q$ be a convex full subquiver in $\olQ $.
If $i$ is a forward movable sink of $Q$, we define the {\em $\tau$-mutation} $ s_i^-Q$ of $Q$ at $i$ as the full bound sub-quiver of $\olQ $ obtained by replacing the vertex $i$ by its inverse $n$-translation $\tau^{-1} i$.
If $i$ is a backward movable source of $Q$, define the {\em $\tau$-mutation} $s_i^+Q$ of $Q$ at $i$ as the full bound sub-quiver in $\olQ $ obtained by replacing the vertex $i$ by its $n$-translation $\tau i$.

\begin{pro}\label{mutation}
If $i$ is a forward movable sink of $Q$, then $s_i^- Q $ is convex in $\olQ$ and $$s_i^+s_i^- Q = Q.$$

If $i$ is a backward movable source of $Q$, then $s_i^+ Q $ is convex in $\olQ$ and $$s_i^-s_i^+ Q = Q.$$
\end{pro}
\begin{proof} We prove the first assertion, the second follows dually.

We need only to prove that $s_i^- Q $  is convex.
Let $p$ be a path in $\olQ$ from $j$ to $j'$ with $j,j'$ in $s^-_iQ$.
If $j' \neq \tau^{-1} i$, then $j,j'$ are both in the full subquiver $Q'$ obtained from $Q$ by removing $i$, so $p$ is also in $Q'$ since $i$ is a source.
Hence p is in $s^{-}_iQ$ since $Q'$ is also a full subquiver of it.
If $j' = \tau^{-1} i$, then $p= \xa q$ for an arrow $\xa$ in $\olQ$ from $j''$ to $j'$ and a path $q$ in $\olQ$ from $j$ to $j''$.
By definition, $j''$ is in $H^i = H_{\tau^{-1}i}$, hence in $Q$, since $i$ is a forward movable sink.
Thus $q$ is in $s^{-}_iQ$, as is proved above.
So $p$ is also in $s^{-}_iQ$.
\end{proof}

If $i$ is a forward movable sink (respectively, a backward movable source), then $s^{-}_i Q$  (respectively, $s^{+}_i Q$) is convex, so we may regard it as a bound quiver with natural relations induced from $\olrho$.
The algebra $\LL(s^{-}_i Q)$ (respectively, $\LL(s^{+}_i Q)$)  of the convex bound subquiver $s^{-}_i Q$ (respectively, $s^{+}_i Q$) is called {\em the $\tau$-mutation} of $\LL$ at $i$, and is denoted  as $s^-_i \LL$ (respectively, $s^{+}_i \LL$).
The quadratic dual $\GG(s^\pm_i Q ) = \LL^{!,op} (s^\pm_i Q)$ of $\LL (s^\pm_i Q)$ is called {\em the $\tau$-mutation} of $\GG$ at $i$, and is denoted by $s^\pm_i \GG$.

\medskip

Now we show that $n$-APR tilts (respectively, cotilts) for a dual truncation of an acyclic $n$-translation algebra are realized by $\tau$-mutation when the vertex is a forward movable sink (respectively, a backward movable source).

\begin{thm}\label{slice-n-apr}
Let $\olL$ be an acyclic $n$-translation algebra with $n$-translation quiver $\olQ$ and $n$-translation $\tau$, let $\olG \simeq k\olQ/(\olrho^{\perp})$ be its quadratic dual.
Assume that $Q = (Q_0,Q_1,\rho)$ is a convex truncation of $\olQ$, then

(1). If $i$ is a forward movable sink of $Q$, let $T$ be the tilting module of $\GG$ related to the Koszul complex \eqref{Koszulcomplex}.
Then $T$ is the $n$-APR tilting module of $\GG$ at $i$ and $$\rend_{\GG} T \simeq s^-_i \GG.$$

(2). If $\tau^{-1}i$ is a backward movable source of $Q$, let $T'$ be the cotilting module of $\GG$ related to the Koszul complex \eqref{Koszulcomplex}.
Then $T'$ is the $n$-APR cotilting module of $\GG$ at $\tau^{-1} i$ and     $$\rend_{\GG} T \simeq s^{+}_{\tau^{-1}i} \GG.$$
\end{thm}
\begin{proof}
We prove the first assertion, the second follows dually.

By Proposition \ref{truncation-alg}, $$\GG \simeq \rend_{\olG} L.$$

Note that for the $n$-translation algebra $\olL$ with $n$-translation $\tau$, $\olL^{op}$ is  $n$-translation algebra with $n$-translation $\tau^{-1}$.
So the Koszul complex of right $\GG$-modules
\small\eqqc{aeq2}{\arr{l}{ 0\longrightarrow e_{ i} \LL_{n} \otimes \Gamma \stackrel{}{\longrightarrow} e_{ i} \LL_{n-1} \otimes \Gamma \stackrel{\xi}{\longrightarrow} \cdots\\ \qquad \qquad \stackrel{}{\longrightarrow} e_{ i}\LL_{1} \otimes \Gamma \stackrel{}{\longrightarrow} e_{ i}\LL_{0} \otimes \Gamma = e_{ i}\Gamma \longrightarrow e_i\GG_0 \lrw 0.}}\normalsize
is the projective resolution of the simple right $\GG$-module $ e_i\GG_0$.
Note that $i$ is a sink of $Q$, thus by  Proposition \ref{KoszulinG}, $e_{i}\LL_n \neq 0$.
Apply $D$, we get and injective resolution of $\GG_0 e_i$:
\small\eqqcn{aeq3}{\arr{l}{
0 \lrw \GG_0 e_i = D(e_i \GG_0 )\longrightarrow D( e_{ i} \LL_{0} \otimes \Gamma) \stackrel{}{\longrightarrow} \cdots \\ \qquad \qquad \stackrel{}{\longrightarrow} D( e_{ i}\LL_{n-1}) \otimes \Gamma \stackrel{}{\longrightarrow} D( e_{ i}\LL_{n} \otimes \Gamma) \longrightarrow  0.
}}\normalsize
This is an injective resolution of the simple $\GG$-module $S(i) \simeq \GG_0 e_i$, and \eqref{aeq2} is the projective resolution of $D(S_i) \simeq e_i \GG_0$.
Applying $\hhm_{\caG}( \quad, \GG)$ to \eqref{aeq2}, one gets:
\small\eqqcn{aeq4}{\arr{l}{
0\longrightarrow\hhom_{\GG}(  e_i\GG_0, \GG)\stackrel{}{\longrightarrow} \hhom_{\GG}(  e_{ i}\LL_{0} \otimes \Gamma,\GG) = \hhom_{\GG}(  e_{ i}\Gamma, \GG) \\ \qquad \stackrel{}{\longrightarrow} \hhom_{\GG}( e_{ i}\LL_{1} \otimes \Gamma,\GG) \stackrel{}{\longrightarrow} \cdots \\ \qquad \qquad\longrightarrow \hhom_{\GG}( e_{ i} \LL_{n-2} \otimes \Gamma, \GG) \stackrel{\xi^*}{\longrightarrow} \hhom_{\GG}( e_{ i} \LL_{n-1} \otimes \Gamma, \GG).
}}\normalsize
Thus $$\cok \xi^* \simeq \ttn{-1} S_i.$$

On the other hand, we have  $$e_{ i} \LL_{t} \otimes \Gamma \simeq \bigoplus\limits_{(j,t) \in H^i_{\olL} } e_j \GG.$$
So by Lemma \ref{kernelKcomplex}, we have the following commutative diagram with isomorphisms between rows:
\small
$$\arr{cccccccc}{
\hhom_{\GG}(\bigoplus\limits_{(j,n-2) \in H^i_{\olL} } e_{ j} \GG, \GG) & \stackrel{\xi^*}{\longrightarrow} &\hhom_{\GG}( \bigoplus\limits_{(j,n-1) \in H^i_{\olL} } e_{ j} \GG, \GG) &  \to &\cok \xi^* &\to &0 \\
\downarrow \simeq &&\downarrow \simeq &&\downarrow \simeq&\\
\bigoplus\limits_{(j,2) \in H^i_{\olL} }\GG e_{ j}  & \stackrel{\xi^*}{\longrightarrow} & \bigoplus\limits_{(j,n-1) \in H^i_{\olL} } \olG e_{ j}  & \to &\cok \xi^* &\to &0 \\
\downarrow \simeq &&\downarrow \simeq &&\downarrow \simeq&\\
\bigoplus\limits_{(j,n-2) \in H^i_{\olL} } e \olG e_j & \stackrel{\xi^*}{\longrightarrow} & \bigoplus\limits_{(j,1) \in H^i_{\olL} } e \olG e_{ j} & \to &\cok \xi^* &\to &0 \\
\downarrow \simeq &&\downarrow \simeq &&\downarrow \simeq&\\
\hhom_{\GG}(\bigoplus\limits_{(j,n-2) \in H^i_{\olL} } e_{ j} \GG, \GG) & \stackrel{}{\longrightarrow} &\hhom_{\caG}( \bigoplus\limits_{(j,n-1) \in H^i_{\olL} } e_{ j} \olG, e\olG) & \to & \hhom_{\caG}(  e_{\tau^{-1} i} \olG, e\olG) &\to &0 .\\
}$$\normalsize
Thus $$\ttn{-1} S_i \simeq \hhom_{\caG}(  e_{\tau^{-1} i} \olG, e\olG) \simeq e\olG e_{\tau^{-1} i}.$$

So  $$T \simeq \cok \xi^* \oplus L\sk{i}(Q) = L(s^-_i Q) $$ is the $n$-APR-tilting module for the simple left $\GG$-projective module $S_i$.
By Proposition \ref{mutation}, $s^-_i Q$ is convex, so we have that $$\rend_{\GG} T \simeq\rend_{\caG} L(s^-_i Q) =  \rend_{\caG}\bigoplus\limits_{j\in (s^-_iQ)_0} \olG e_j \simeq k(s^-_i Q)/(s^-_i \rho^{\perp}) = s^-_i \GG,$$
by Lemma \ref{truncation-alg}.
\end{proof}

This shows that for a dual  truncation  algebra of an acyclic $n$-translation quiver, the $n$-APR  tilts for a forward movable sink (respectively, cotilts for a backward movable source) is realized by the $\tau$-mutation of its bound quiver.

\section{Application to dual $\tau$-slice algebras}\label{slices}

In \cite{io11}, $n$-APR tilting complexes of an $n$-representation finite algebra are related to the slices and mutations in  $n$-cluster tilting  subcategory of the derived category.
In \cite{g12}, we introduce the $\tau$-slice algebras of a given graded self-injective algebra and  showed that they are derived equivalent by showing that they have isomorphic trivial extensions.
Using our results in the previous section, we also have such equivalences for dual $\tau$-slice algebras, and the $\tau$-mutations for such algebras are explained as  $n$-APR tilts here.

Let $\olQ=(\olQ_0, \olQ_1, \olrho)$ be an acyclic stable $n$-translation quiver with $n$-translation $\tau$, and assume that $\olQ$ has only finite many $\tau$-orbits.
Let $Q$ be a full sub-quiver of $\olQ$.
$Q$ is called a  {\em complete $\tau$-slice} of $\olQ$ if it is convex (called path complete in \cite{g12}) and for each vertex $v $ of $\olQ$, the intersection of the $\tau$-orbit of $v$ and the vertex set of $Q$ is a single-point set.
The following lemma is obvious.
\begin{lem}\label{slicemovable}
If $Q$ is a complete $\tau$-slice in a stable $n$-translation quiver $\olQ$, then its sinks are forward movable and its sources are backward movable.
\end{lem}

We usually take the relation set $\olrho$ is normalized such that $$\rho = \{x = \sum\limits_p a_p p \in \olrho| s(p), t(p) \in Q_0 \} \subseteq \olrho.$$
So the complete $\tau$-slice $Q$ is regarded as a bound quiver $Q=(Q_0,Q_1,\rho)$.

The algebra $\LL$ defined by a complete $\tau$-slice $Q$ in  $\olQ$ is called a {\em $\tau$-slice algebra} of the bound quiver $\olQ$.
Obviously, we have the following consequences.
\begin{lem}\label{slicetrunct}
Let $\olQ$  be a stable $n$-translation quiver. Then its complete $\tau$-slices  are  convex truncations.
\end{lem}

If $\olQ$ is the bound quiver of an algebra $\olL$, we also say that $\LL$ is a {\em $\tau$-slice algebra} of $\olL$.
If $\olL$ is an $n$-translation algebra, $\LL$ is a quadratic algebra, and we call its quadratic dual $\GG= {\LL}^{!op}$ the {\em dual $\tau$-slice algebra}.

We have shown in \cite{g12} that $\tau$-slices are related by $\tau$-mutations, as in the following lemma.

\begin{pro}\label{slicemutation}
Let $Q$ be a complete $\tau$-slice of a stable $n$-translation quiver $\olQ$.

If $i$ is a sink of $Q$, then $s^-_iQ$ is a complete $\tau$-slice of $\olQ$.

If $i$ is a source of $Q$, then $s^+_iQ$ is a complete $\tau$-slice of $\olQ$.

If $Q$, $Q'$ are two complete $\tau$-slices in $\olQ$, then there is a sequence $s_{i_1}^{*_1}, \ldots, s_{i_r}^{*_r}$, where $i_t$ are vertices in $\olQ$ and $*_t\in \{+,-\}$, such that $$Q' =s_{i_r}^{*_r}\cdots s_{i_1}^{*_1}Q.$$
\end{pro}

We remark that in Proposition \ref{slicemutation}, we may take all the vertices $i_1, \ldots, i_r$ to be the sinks in the corresponding quivers, and the mutations as $s_{i_1}^{-}, \ldots, s_{i_r}^{-}$, or all to be the sources and the mutations as  $s_{i_1}^{+}, \ldots, s_{i_r}^{+}$.

\medskip

If $Q$ is a complete $\tau$-slice of $\olQ$ and $i$ is a sink (respectively, $\tau^{-1} i$ is a source) of $Q$, we know that the algebra $\LL(Q)\simeq kQ/(\rho)$ and its $\tau$-mutation $\LL(s_i^-Q)\simeq ks_i^-Q/(s_i^-\rho)$ (respectively, $\LL(s_{\tau^{-1} i}^+Q)\simeq ks_{\tau^{-1} i}^+Q/(s_{\tau^{-1} i}^+\rho)$) are derived equivalent (see Corollary 6.11 of \cite{g12}).
Using  Proposition \ref{truncation-alg}, Lemma \ref{slicemovable} and Proposition \ref{slicemutation}, we  have the following refinement of Theorem \ref{tilting} for a dual $\tau$-slice algebra.

\begin{cor}\label{slice-tilting}
Assume that $\olL$ is an $n$-translation algebra with bound quiver $\olQ$, and $Q$ is a complete $\tau$-slice.
Let $\LL = \LL(Q)$, $\GG= \GG(Q)$.
Then

(1). If $i$ is a sink of $Q$, let $T$ be the tilting module of $\GG$ related to the Koszul complex \eqref{Koszulcomplex}, then $T$ is an $n$-APR tilting module, $\rend_{\GG}T  $ is a dual $\tau$-slice algebra and $$\rend_{\GG}T  \simeq s_i^- \GG.$$

(2). If $\tau^{-1} i$ is a source of $Q$, let $T'$ be the cotilting module of $\GG$ related to the Koszul complex \eqref{Koszulcomplex}, then $T'$ is an $n$-APR cotilting module, $\rend_{\GG}T  $ is a dual $\tau$-slice algebra and $$\rend_{\GG}T \simeq s_i^+ \GG.$$
\end{cor}

So we see for the dual $\tau$-slice algebra of an acyclic stable $n$-translation algebra, $n$-APR tilts and cotilts are realized by $\tau$-mutations, and verse visa.

\medskip

Let $\LL$ be a $\tau$-slice algebra with bound quiver $Q =(Q_0,Q_1,\rho)$ which is a $\tau$-slice of a stable $n$-translation quiver $\olQ$.
Now we show that $\olQ$ can be take as the quiver $\zzs{n-1}Q$ defined in \cite{g16}.

By Lemma 6.1 of \cite{g20}, $Q$ is $n$-properly-graded quiver and maximal bound paths of $Q$ have the same length $n$.
Let $\caM$ be a set of linearly independent maximal bound paths in $Q$.
Define {\em returning arrow quiver} $\tQ= (\tQ_0,\tQ_1,\trho)$ with $$\tQ_0 = Q_0, \quad \tQ_1 = Q_1\cup Q_{1,\caM}, $$ where $$Q_{1,\caM} = \{\beta_{p}: t(p) \to s(p)|  p\in \caM\}.$$
So $\tQ$ is obtained from $Q$ by adding an arrow in the reversed direction to each maximal bound path in $Q$.

Denote $\dtL = \LL\ltimes D\LL$  the  trivial extension of $\LL$, the  returning arrow quiver of $Q$ is exactly the bound quiver of $\dtL = \LL\ltimes D\LL$ (Proposition 2.2 of \cite{fp02}).
\begin{pro}\label{quiverTE}
If $Q=(Q_0,Q_1,\rho)$ is the bound quiver of a $\tau$-slice algebra $\LL$, then there is a relation set $\trho$ such that $\tQs = (\tQ_0, \tQ_1, \trho)$ is the bound quiver of the trivial extension  $\dtL$ of $\LL$.

$\trho$ is quadratic if $\rho$ is so.
\end{pro}

For a complete $\tau$-slice $Q$,
Recall that we can constructed the $\zZ|_{n-1} Q$ for $Q$ in \cite{g16}.
Take vertex set $$(\zZ|_{n-1} Q)_0 =\{(i , t)| i\in Q_0, t \in \zZ\},$$ arrow set $$\arr{rl}{(\zZ|_{n -1}Q)_1  = &\zZ \times Q_1 \cup \zZ \times \caM^{op}\\ = & \{(\alpha,t): (i,t)\longrightarrow (j,t) | \alpha:i\longrightarrow j \in Q_1, t \in \zZ\}\\ &\quad \cup \{(\beta_p , t): (j, t) \longrightarrow (i, t+1) | p\in \caM, s(p)=i,t(p)=j  \}}$$ and relation set $$\rho_{\zZ|_{n-1} Q} =\zZ \rho\cup \zZ \rho_{\caM} \cup \zZ\rho_0,$$ where
$$\zZ \rho =  \{\sum_{s} a_s (\xa_s,t)\otimes (\xa'_s,t) |\sum_{s} a_s \xa_s\otimes \xa'_s \in \rho, t\in \zZ\},$$ $$\zZ \rho_{\caM} =  \{(\beta_{p'},t+1) \otimes (\beta_p ,t)| \beta_{p'}\otimes \beta_{p}\in \rho_{\caM}, t\in \zZ\}$$ and $$\arr{rl}{\zZ \rho_0 = & \{ \sum_{s'} a_{s'} (\beta_{p'_{s'}},t+1)\otimes  (\xa'_{s'}, t) + \sum_{s} b_s (\xa_s,t)\otimes  (\beta_{p_s} ,t) \\ &\qquad |  \sum_{s'} a_{s'} \beta_{p'_{s'}}\otimes \xa'_{s'} + \sum_{s} b_s \xa_s,t\otimes \beta_{p_s} \in \rho_0 , t\in \zZ\}.}$$

Similar to Proposition 5.5 of \cite{g16}, we have the following realization of $\zZ|_{n-1} Q$.

\begin{pro}\label{0nap:extendible1}
Let $\LL$  be an algebra as defined in Proposition \ref{quiverTE} such that $\dtL$ is quadratic.
Then the smash product $\dtL\#  k \zZ^*$ is a self-injective algebra with bound quiver $\zZ|_{n-1} Q $, where $\dtL$ is graded by taking elements in the dual basis of $\caM$ in $D\LL_n$ as degree $1$ generators.
\end{pro}

Since $Q$ is acyclic, it is a complete $\tau$-slice in $\zZ|_{n-1} Q$, so $\LL$ is a $\tau$-slice algebra of $\dtL\#  k \zZ^*$.

As a corollary of Corollary \ref{slice-tilting} and Proposition \ref{slicemutation}, we have the following corollary.
\begin{cor}\label{slice-tilting_it}
Let $\GG$ be a finite dimensional connected dual $\tau$-slice algebra of an acyclic stable $n$-translation algebra. There are only finitely many algebras obtained from $\GG$ using iterated $n$-APR tilts and cotilts.
\end{cor}
\begin{proof}
Assume that the bound quiver of $\GG$ is $Q^{\perp}$, then $Q$ is a $\tau$-slice of   $\zzs{n-1} Q$ and the iterated $n$-APR tilts and cotilts are  dual $\tau$-slice algebras of $\zzs{n-1}Q$, by Corollary \ref{slice-tilting}.
But $\tau$-slices  of $\zzs{n-1}Q$ are connected and convex, so there are only finitely many up to shifted by the $n$-translation $\tau$.
This shows that up to isomorphism, there are only finitely many algebras obtained from $\GG$ using iterated $n$-APR tilts and cotilts.
\end{proof}

Now we have the following algorithm to construction $n$-APR tilts and cotilts for dual $\tau$-slice algebras.
Let $\GG$  be a dual $\tau$-slice algebra with bound quiver $Q^{\perp}= (Q_0, Q_1, \rho^{\perp})$.
Let $\rho$ be a basis of the orthogonal subspace of $\rho^{\perp}$ in $kQ_2$, then $Q= (Q_0, Q_1, \rho)$ is the bound quiver of $\tau$-slice algebra.
Construct $\zzs{n-1}Q$ as above, $\zzs{n-1}Q$ is the bound quiver of a stable $n$-translation algebra $\olL$.
$Q$ is a $\tau$-slice in $\zzs{n-1}Q$, we in fact recovered the stable $n$-translation quiver $\olQ$.
For each sink $i$ in $Q$, take the $\tau$-mutation $s^-_i Q$ in $\zzs{n-1}Q$, we obtained the $n$-APR tilts $s^-_i\GG = \GG(s^-_i Q)$ of $\GG$ with respect to the simple projective $\GG$-module $\GG_0  e_i$.
For each source $i$ in $Q$, take the $\tau$-mutation $s^+_i Q$ in $\zzs{n-1}Q$, we obtained the $n$-APR cotilts $s^+_i \GG =\GG(s^+_i Q)$ of $\GG$ with respect to the simple injective $\GG$-module $\GG_0  e_i$.

\medskip

\begin{exa}\label{E:one}%{\em
In \cite{io11}, iterated $n$-APR tilts of an $n$-representation-finite algebra of type $A$ are characterization using mutations on cuts.
Now we show by example how we get the iterated $2$-APR tilts of a $2$-representation-finite algebra of type $A$  using $\tau$-mutations on $\tau$-slices.

The Auslander algebra $\GG = \GG(2)$ of the path algebra $\GG(1)$ of type $A_3$ with linear orientation, is a $2$-representation-finite algebra, given by the quiver $Q^{\perp}(2)$:
$$
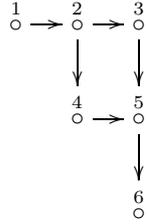
\xymatrix@C=0.4cm@R0.6cm{
&& \stackrel{1}{\circ} \ar[r] &\stackrel{2}{\circ} \ar[r]\ar[d] &\stackrel{3}{\circ}\ar[d]\\
&&              &\stackrel{4}{\circ} \ar[r]       &\stackrel{5}{\circ}\ar[d] &{} \\
&&                                 &&\stackrel{6}{\circ} &{} \\
}
$$
with the returning arrow quiver $\tQ^{\perp}(2)$,
$$
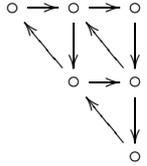
\xymatrix@C=0.4cm@R0.6cm{
&& \circ \ar[r] &\circ \ar[r]\ar[d] &\circ\ar[d] &{}\\
&&              &\circ \ar[r] \ar[ul]      &\circ\ar[d] \ar[ul] &{} \\
&&                                 &&\circ \ar[ul] &{} \\
}.
$$
This is also the quiver of the (twisted) preprojective algebra $\GG(2)$.
The quiver $\zzs{1}Q^{\perp}$ is as following

$$
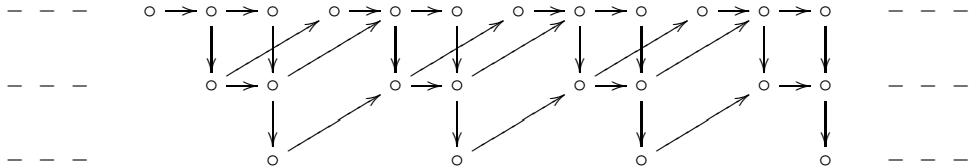
\xymatrix@C=0.4cm@R0.6cm{
&&\ar@{--}[ll]& \circ \ar[r] &\circ \ar[r]\ar[d] &\circ\ar[d] &{}  \circ \ar[r] &\circ \ar[r]\ar[d] &\circ\ar[d] &{} \circ \ar[r] &\circ \ar[r]\ar[d] &\circ\ar[d] &{} \circ \ar[r] &\circ \ar[r]\ar[d] &\circ\ar[d] &\ar@{--}[rr]&&\\
&& \ar@{--}[ll]  &           &\circ \ar[r] \ar[urr]      &\circ\ar[d]\ar[urr] &{}              &\circ \ar[r]\ar[urr]       &\circ\ar[d]\ar[urr] &{}              &\circ \ar[r] \ar[urr]      &\circ\ar[d]\ar[urr] &{}              &\circ \ar[r]      &\circ\ar[d] &\ar@{--}[rr]&&\\
&& \ar@{--}[ll] &                               &&\circ\ar[urr] &{}                                 &&\circ\ar[urr] &{}                                 &&\circ\ar[urr] &{}                                 &&\circ &\ar@{--}[rr]&&\\
}
$$
The hammocks
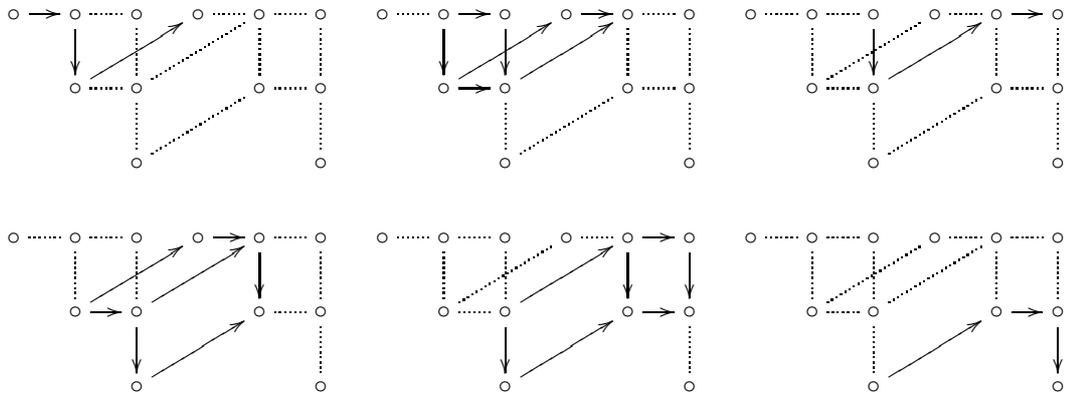
$$
\xymatrix@C=0.4cm@R0.6cm{
\circ \ar[r] &\circ \ar@{.}[r]\ar[d] &\circ\ar@{.}[d] &{}  \circ \ar@{.}[r] &\circ \ar@{.}[r]\ar@{.}[d] &\circ\ar@{.}[d] &{}{} \circ \ar@{.}[r] &\circ \ar[r]\ar[d] &\circ\ar[d] &{} \circ \ar[r] &\circ \ar@{.}[r]\ar@{.}[d] &\circ\ar@{.}[d] &{}{} \circ \ar@{.}[r] &\circ \ar@{.}[r]\ar@{.}[d] &\circ\ar[d] &{} \circ \ar@{.}[r] &\circ \ar[r]\ar@{.}[d] &\circ\ar@{.}[d]&&\\
             &\circ \ar@{.}[r] \ar[urr]      &\circ\ar@{.}[d]\ar@{.}[urr] &{}              &\circ \ar@{.}[r]     &\circ\ar@{.}[d] &{}{}              &\circ \ar[r] \ar[urr]      &\circ\ar@{.}[d]\ar[urr] &{}              &\circ \ar@{.}[r]      &\circ\ar@{.}[d] &{}{}              &\circ \ar@{.}[r] \ar@{.}[urr]      &\circ\ar@{.}[d]\ar[urr] &{}              &\circ \ar@{.}[r]      &\circ\ar@{.}[d] &&\\
                                &&\circ\ar@{.}[urr] &{}                                 &&\circ &{}{}                                 &&\circ\ar@{.}[urr] &{}                                 &&\circ &{}{}                                 &&\circ\ar@{.}[urr] &{}                                 &&\circ&&\\%first
\circ \ar@{.}[r] &\circ \ar@{.}[r]\ar@{.}[d] &\circ\ar@{.}[d] &{}  \circ \ar[r] &\circ \ar@{.}[r]\ar[d] &\circ\ar@{.}[d]  &{}{} \circ \ar@{.}[r] &\circ \ar@{.}[r]\ar@{.}[d] &\circ\ar@{.}[d] &{} \circ \ar@{.}[r] &\circ \ar[r]\ar[d] &\circ\ar[d] &{} \circ \ar@{.}[r] &\circ \ar@{.}[r]\ar@{.}[d] &\circ\ar@{.}[d] &{} \circ \ar@{.}[r] &\circ \ar@{.}[r]\ar@{.}[d] &\circ\ar@{.}[d]&&\\
             &\circ \ar[r] \ar[urr]      &\circ\ar[d]\ar[urr] &{}              &\circ \ar@{.}[r]     &\circ\ar@{.}[d] &{}{}              &\circ \ar@{.}[r] \ar@{.}[urr]      &\circ\ar[d]\ar[urr] &{}              &\circ \ar[r]      &\circ\ar@{.}[d] &{}{}              &\circ \ar@{.}[r] \ar@{.}[urr]      &\circ\ar@{.}[d]\ar@{.}[urr] &{}              &\circ \ar[r]      &\circ\ar[d] &&\\
                                &&\circ\ar[urr] &{}                                 &&\circ &{}{}                                 &&\circ\ar[urr] &{}                                 &&\circ &{}{}                                 &&\circ\ar[urr] &{}                                 &&\circ&&\\%second
}
$$

The complete $\tau$-slices, or the quivers of iterated $2$-APR tilts and cotilts of $\GG(2)$ are obtained by iterated $\tau$-mutations.
The $\tau$-mutation $s^+_i$ with respect to a source $i$ is obtained by removing the source of hammock $H^i$ and adding the sink $\tau^{-1} i$ with the arrows to $\tau^{-1}i$, as is shown below.
$$\arr{c}
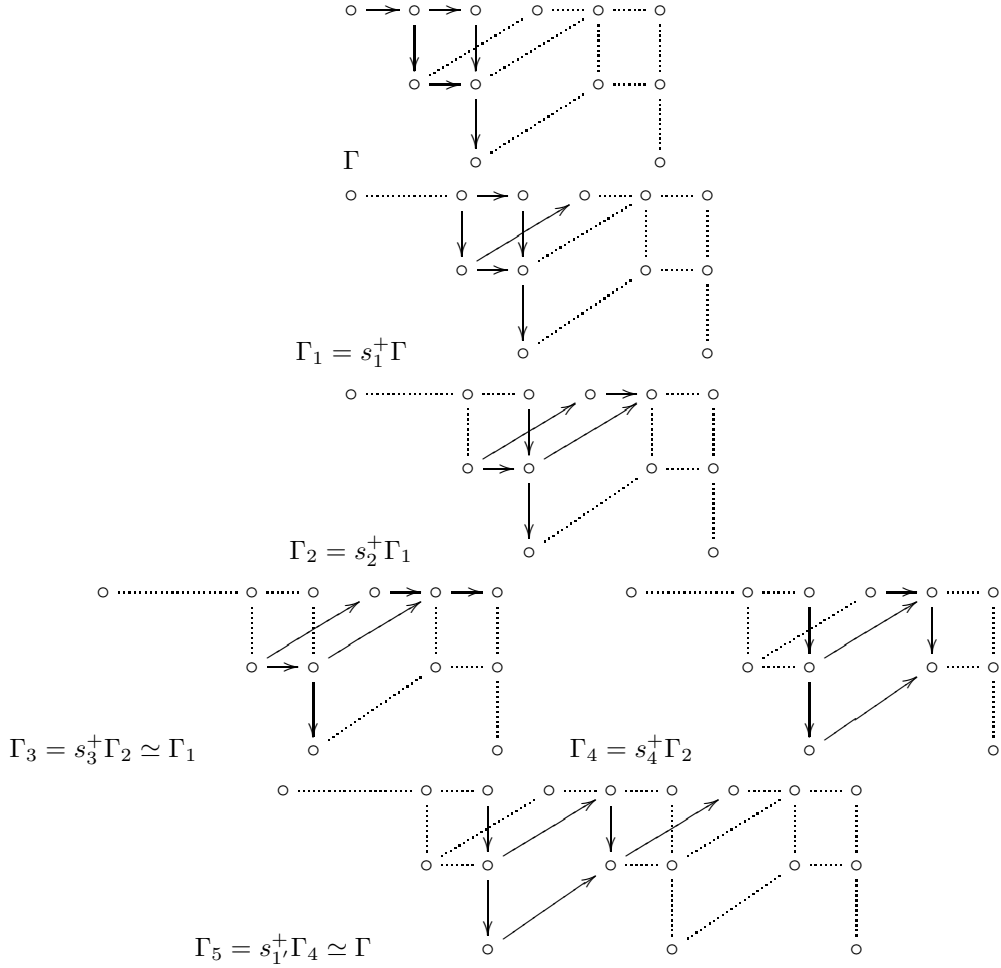
{
\xymatrix@C=0.4cm@R0.6cm{
\circ \ar[r] &\circ \ar[r]\ar[d] &\circ\ar[d] &{}  \circ \ar@{.}[r] &\circ \ar@{.}[r]\ar@{.}[d] &\circ\ar@{.}[d] &{}{}\\
             &\circ \ar[r] \ar@{.}[urr]      &\circ\ar[d]\ar@{.}[urr] &{}              &\circ \ar@{.}[r]     &\circ\ar@{.}[d] &{}{} \\
                                \GG&&\circ\ar@{.}[urr] &{}                                 &&\circ &{}{}  \\}\\%first
%2
\xymatrix@C=0.4cm@R0.6cm{
\circ \ar@{.}[r] &\circ \ar[r]\ar[d] &\circ\ar[d] &{}  \circ \ar@{.}[r] &\circ \ar@{.}[r]\ar@{.}[d] &\circ\ar@{.}[d] &{}{}\\
             &\circ \ar[r] \ar[urr]      &\circ\ar[d]\ar@{.}[urr] &{}              &\circ \ar@{.}[r]     &\circ\ar@{.}[d] &{}{}\\
                               \GG_1= s^+_1 \GG&&\circ\ar@{.}[urr] &{}                                 &&\circ &{}{}  \\}\\%first
%3
\xymatrix@C=0.4cm@R0.6cm{
\circ \ar@{.}[r] &\circ \ar@{.}[r]\ar@{.}[d] &\circ\ar[d] &{}  \circ \ar[r] &\circ \ar@{.}[r]\ar@{.}[d] &\circ\ar@{.}[d] &{}{}\\
             &\circ \ar[r] \ar[urr]      &\circ\ar[d]\ar[urr] &{}              &\circ \ar@{.}[r]     &\circ\ar@{.}[d] &{}{}\\
                                \GG_2= s^+_2  \GG_1&&\circ\ar@{.}[urr] &{}                                 &&\circ &{}{}  \\}\\%first
%4
\xymatrix@C=0.4cm@R0.6cm{
\circ \ar@{.}[r] &\circ \ar@{.}[r]\ar@{.}[d] &\circ\ar@{.}[d] &{}  \circ \ar[r] &\circ \ar[r]\ar@{.}[d] &\circ\ar@{.}[d] &{}{}\\
             &\circ \ar[r] \ar[urr]      &\circ\ar[d]\ar[urr] &{}              &\circ \ar@{.}[r]     &\circ\ar@{.}[d] &{}{}\\
                              \GG_3= s^+_3  \GG_2 \simeq \GG_1 &&\circ\ar@{.}[urr] &{}                                 &&\circ &{}{}  \\}
%5
\xymatrix@C=0.4cm@R0.6cm{
\circ \ar@{.}[r] &\circ \ar@{.}[r]\ar@{.}[d] &\circ\ar[d] &{}  \circ \ar[r] &\circ \ar@{.}[r]\ar[d] &\circ\ar@{.}[d] &{}{}\\
             &\circ \ar@{.}[r] \ar@{.}[urr]      &\circ\ar[d]\ar[urr] &{}              &\circ \ar@{.}[r]     &\circ\ar@{.}[d] &{}{}\\
                            \GG_4= s^+_4  \GG_2    &&\circ\ar[urr] &{}                                  &&\circ &{}{}  \\}\\
%6
\xymatrix@C=0.4cm@R0.6cm{
\circ \ar@{.}[r] &\circ \ar@{.}[r]\ar@{.}[d] &\circ\ar[d] &{}  \circ \ar@{.}[r] &\circ \ar@{.}[r]\ar[d] &\circ\ar@{.}[d] &{}{}  \circ \ar@{.}[r] &\circ \ar@{.}[r]\ar@{.}[d] &\circ\ar@{.}[d] \\
             &\circ \ar@{.}[r] \ar@{.}[urr]      &\circ\ar[d]\ar[urr] &{}              &\circ \ar@{.}[r] \ar[urr]    &\circ\ar@{.}[d]\ar@{.}[urr] &{}{} &\circ \ar@{.}[r]     &\circ\ar@{.}[d] \\
                               \GG_5= s^+_{1'}  \GG_4\simeq \GG &&\circ\ar[urr] &{}                                 &&\circ\ar@{.}[urr] &{}{} &&\circ  \\}\\
%7
}
$$
These are exactly the quivers listed in Table 1 of \cite{io11}.

%}
\end{exa}

\begin{exa}\label{E:two}%{\em
By \cite{m79}, the McKay quiver of a finite subgroup of $\mathrm{SL}(\mathbb C^2)$ is a double quiver of extended Dynkin diagram.
Fix $G$ with McKay quiver \eqqcn{D4double}
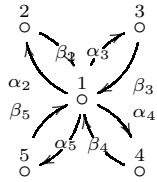
{ \xymatrix@C=0.35cm@R0.3cm{
&&\stackrel{2}{\circ}\ar@/^/[dr]|{\beta_2}&&\stackrel{3}{\circ}\ar@/^/[dl]^{\beta_3}&&&\\
&& & \stackrel{1}{\circ}\ar@/^/[lu]^{\alpha_2}\ar@/^/[ld]|{\alpha_5} \ar@/^/[ur]|{\alpha_3}\ar@/^/[dr]^{\alpha_4}& &&\\
&&\stackrel{5}{\circ}\ar@/^/[ur]^{\beta_5}&&\stackrel{4}{\circ}\ar@/^/[ul]|{\beta_4}&&&\\
}}
Embedding $G$ in $ \mathrm{SL}_3(\mathbb C)$ in a natural way, the new McKay quiver for $G$ in $3$-dimensional space is the returning arrow quiver $\tQ$, by Theorem 3.1 of \cite{g11}.
\eqqcn{D4doublereturning}
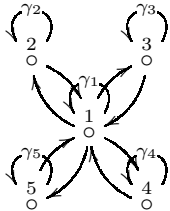
{ \xymatrix@C=0.35cm@R0.3cm{
&&\stackrel{2}{\circ}\ar@/^/[dr]\ar@(ur,ul)[]|{\gamma_2}&&\stackrel{3}{\circ}\ar@/^/[dl]\ar@(ur,ul)[]|{\gamma_3}&&&\\
&& & \stackrel{1}{\circ}\ar@(ur,ul)[]|{\gamma_1}\ar@/^/[lu]\ar@/^/[ld] \ar@/^/[ur]\ar@/^/[dr]& &&\\
&&\stackrel{5}{\circ}\ar@(ur,ul)[]|{\gamma_5}\ar@/^/[ur]&&\stackrel{4}{\circ}\ar@(ur,ul)[]|{\gamma_4}\ar@/^/[ul]&&&\\
}}
with relations as described in Proposition 2.5 of \cite{gyz14}.
By Corollary 3.2  of \cite{gum}, this is a stable $2$-translation quiver with trivial $2$-translation, associated to a Koszul $2$-translation algebra $\tL(G)$, which is Morita equivalent to the skew group algebra $(\wedge \mathbb C^3) * G$ by Theorem 2.1 of \cite{gum}.

We can construct a Koszul $2$-translation algebra $\olL(G)$, with acyclic stable $2$-translation quiver $\olQ$:
\eqqcn{Zquiver}
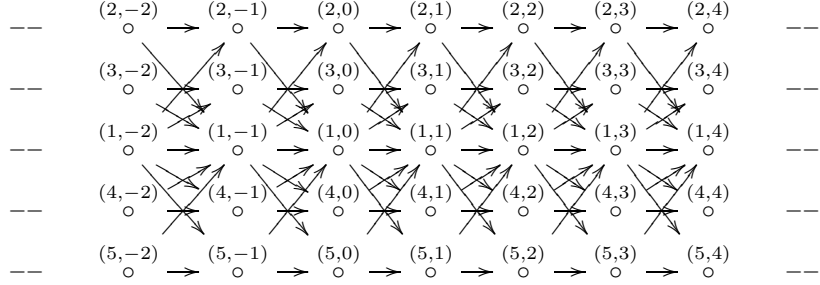
{\xymatrix@C=0.4cm@R0.1cm{
\ar@{--}[r]&& \stackrel{(2,-2)}{\circ} \ar[ddr]\ar[r] &\stackrel{(2,-1)}{\circ}\ar[ddr]\ar[r]
&\stackrel{(2,0)}{\circ} \ar[ddr]\ar[r]& \stackrel{(2,1)}{\circ} \ar[ddr]\ar[r] &\stackrel{(2,2)}{\circ}\ar[ddr]\ar[r]
&\stackrel{(2,3)}{\circ} \ar[ddr]\ar[r]& \stackrel{(2,4)}{\circ} &&\ar@{--}[l]  \\
\ar@{--}[r]&& \stackrel{(3,-2)}{\circ} \ar[dr]\ar[r] &\stackrel{(3,-1)}{\circ}\ar[dr]\ar[r]
&\stackrel{(3,0)}{\circ} \ar[dr]\ar[r]& \stackrel{(3,1)}{\circ} \ar[dr]\ar[r] &\stackrel{(3,2)}{\circ}\ar[dr]\ar[r]
&\stackrel{(3,3)}{\circ} \ar[dr]\ar[r]& \stackrel{(3,4)}{\circ} &&\ar@{--}[l]  \\
\ar@{--}[r]&& \stackrel{(1,-2)}{\circ}\ar[ddr] \ar[dr] \ar[ur]\ar[uur]\ar[r]&\stackrel{(1,-1)}{\circ}\ar[ddr] \ar[dr] \ar[ur]\ar[uur]\ar[r]
&\stackrel{(1,0)}{\circ}\ar[ddr] \ar[dr] \ar[ur]\ar[uur]\ar[r]& \stackrel{(1,1)}{\circ}\ar[ddr] \ar[dr] \ar[ur]\ar[uur]\ar[r]&\stackrel{(1,2)}{\circ}\ar[ddr] \ar[dr] \ar[ur]\ar[uur]\ar[r]
&\stackrel{(1,3)}{\circ}\ar[ddr] \ar[dr] \ar[ur]\ar[uur]\ar[r]& \stackrel{(1,4)}{\circ}&&\ar@{--}[l]\\
\ar@{--}[r]& &  \stackrel{(4,-2)}{\circ} \ar[ur]\ar[r]& \stackrel{(4,-1)}{\circ}\ar[ur]\ar[r]
&\stackrel{(4,0)}{\circ} \ar[ur]\ar[r]&  \stackrel{(4,1)}{\circ} \ar[ur]\ar[r]& \stackrel{(4,2)}{\circ}\ar[ur]\ar[r]
&\stackrel{(4,3)}{\circ} \ar[ur]\ar[r]&  \stackrel{(4,4)}{\circ} &&\ar@{--}[l]  \\
\ar@{--}[r]&& \stackrel{(5,-2)}{\circ}  \ar[uur]\ar[r] &\stackrel{(5,-1)}{\circ}\ar[uur]\ar[r]
&\stackrel{(5,0)}{\circ}\ar[uur]\ar[r]& \stackrel{(5,1)}{\circ}  \ar[uur]\ar[r] &\stackrel{(5,2)}{\circ}\ar[uur]\ar[r]
&\stackrel{(5,3)}{\circ}\ar[uur]\ar[r]& \stackrel{(5,4)}{\circ}  &&\ar@{--}[l]  \\
}
}
with the $2$-translation $\tau$ sending vertex $(i,t)$ to $(i,t-3)$. The $\tau$-hammocks are of the form
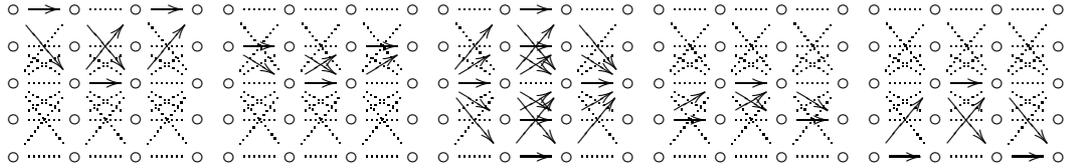
\eqqcn{}{
\arr{c}{
\xymatrix@C=0.4cm@R0.1cm{
 {\circ} \ar[ddr]\ar[r] &{\circ}\ar[ddr]\ar@{.}[r] & {\circ} \ar@{.}[ddr]\ar[r]& {\circ}  \\
{\circ} \ar@{.}[dr]\ar@{.}[r] &{\circ}\ar@{.}[dr]\ar@{.}[r]&{\circ} \ar@{.}[dr]\ar@{.}[r]& {\circ} \\
{\circ}\ar@{.}[ddr] \ar@{.}[dr] \ar@{.}[ur]\ar@{.}[uur]\ar@{.}[r] &{\circ}\ar@{.}[ddr] \ar@{.}[dr] \ar@{.}[ur]\ar[uur]\ar[r] &{\circ}\ar@{.}[ddr] \ar@{.}[dr] \ar@{.}[ur]\ar[uur]\ar@{.}[r]& {\circ}\\
{\circ} \ar@{.}[ur]\ar@{.}[r]& {\circ}\ar@{.}[ur]\ar@{.}[r] &{\circ} \ar@{.}[ur]\ar@{.}[r]&  {\circ} \\
{\circ}  \ar@{.}[uur]\ar@{.}[r] &{\circ}\ar@{.}[uur]\ar@{.}[r]
&{\circ}\ar@{.}[uur]\ar@{.}[r]& {\circ} \\
}%1
\xymatrix@C=0.4cm@R0.1cm{
 {\circ} \ar@{.}[ddr]\ar@{.}[r] &{\circ}\ar@{.}[ddr]\ar@{.}[r] & {\circ} \ar@{.}[ddr]\ar@{.}[r]& {\circ}  \\
{\circ} \ar[dr]\ar[r] &{\circ}\ar[dr]\ar@{.}[r]&{\circ} \ar@{.}[dr]\ar[r]& {\circ} \\
{\circ}\ar@{.}[ddr] \ar@{.}[dr] \ar@{.}[ur]\ar@{.}[uur]\ar@{.}[r] &{\circ}\ar@{.}[ddr] \ar@{.}[dr] \ar[ur]\ar@{.}[uur]\ar[r] &{\circ}\ar@{.}[ddr] \ar@{.}[dr] \ar[ur]\ar@{.}[uur]\ar@{.}[r]& {\circ}\\
{\circ} \ar@{.}[ur]\ar@{.}[r]& {\circ}\ar@{.}[ur]\ar@{.}[r] &{\circ} \ar@{.}[ur]\ar@{.}[r]&  {\circ} \\
{\circ}  \ar@{.}[uur]\ar@{.}[r] &{\circ}\ar@{.}[uur]\ar@{.}[r]
&{\circ}\ar@{.}[uur]\ar@{.}[r]& {\circ} \\
}%\\%2
\xymatrix@C=0.4cm@R0.1cm{
 {\circ} \ar@{.}[ddr]\ar@{.}[r] &{\circ}\ar[ddr]\ar[r] & {\circ} \ar[ddr]\ar@{.}[r]& {\circ}  \\
{\circ} \ar@{.}[dr]\ar@{.}[r] &{\circ}\ar[dr]\ar[r]&{\circ} \ar[dr]\ar@{.}[r]& {\circ} \\
{\circ}\ar[ddr] \ar[dr] \ar[ur]\ar[uur]\ar[r] &{\circ}\ar[ddr] \ar[dr] \ar[ur]\ar[uur]\ar@{.}[r] &{\circ}\ar@{.}[ddr] \ar@{.}[dr] \ar@{.}[ur]\ar@{.}[uur]\ar[r]& {\circ}\\
{\circ} \ar@{.}[ur]\ar@{.}[r]& {\circ}\ar[ur]\ar[r] &{\circ} \ar[ur]\ar@{.}[r]&  {\circ} \\
{\circ}  \ar@{.}[uur]\ar@{.}[r] &{\circ}\ar[uur]\ar[r]
&{\circ}\ar[uur]\ar@{.}[r]& {\circ} \\
}%3
\xymatrix@C=0.4cm@R0.1cm{
 {\circ} \ar@{.}[ddr]\ar@{.}[r] &{\circ}\ar@{.}[ddr]\ar@{.}[r] & {\circ} \ar@{.}[ddr]\ar@{.}[r]& {\circ}  \\
{\circ} \ar@{.}[dr]\ar@{.}[r] &{\circ}\ar@{.}[dr]\ar@{.}[r]&{\circ} \ar@{.}[dr]\ar@{.}[r]& {\circ} \\
{\circ}\ar@{.}[ddr] \ar@{.}[dr] \ar@{.}[ur]\ar@{.}[uur]\ar@{.}[r] &{\circ}\ar@{.}[ddr] \ar[dr] \ar@{.}[ur]\ar@{.}[uur]\ar[r] &{\circ}\ar@{.}[ddr] \ar[dr] \ar@{.}[ur]\ar@{.}[uur]\ar@{.}[r]& {\circ}\\
{\circ} \ar[ur]\ar[r]& {\circ}\ar[ur]\ar@{.}[r] &{\circ} \ar@{.}[ur]\ar[r]&  {\circ} \\
{\circ}  \ar@{.}[uur]\ar@{.}[r] &{\circ}\ar@{.}[uur]\ar@{.}[r]
&{\circ}\ar@{.}[uur]\ar@{.}[r]& {\circ} \\
} %4
\xymatrix@C=0.4cm@R0.1cm{
 {\circ} \ar@{.}[ddr]\ar@{.}[r] &{\circ}\ar@{.}[ddr]\ar@{.}[r] & {\circ} \ar@{.}[ddr]\ar@{.}[r]& {\circ}  \\
{\circ} \ar@{.}[dr]\ar@{.}[r] &{\circ}\ar@{.}[dr]\ar@{.}[r]&{\circ} \ar@{.}[dr]\ar@{.}[r]& {\circ} \\
{\circ}\ar@{.}[ddr] \ar@{.}[dr] \ar@{.}[ur]\ar@{.}[uur]\ar@{.}[r] &{\circ}\ar[ddr] \ar@{.}[dr] \ar@{.}[ur]\ar@{.}[uur]\ar[r] &{\circ}\ar[ddr] \ar@{.}[dr] \ar@{.}[ur]\ar@{.}[uur]\ar@{.}[r]& {\circ}\\
{\circ} \ar@{.}[ur]\ar@{.}[r]& {\circ}\ar@{.}[ur]\ar@{.}[r] &{\circ} \ar@{.}[ur]\ar@{.}[r]&  {\circ} \\
{\circ}  \ar[uur]\ar[r] &{\circ}\ar[uur]\ar@{.}[r]
&{\circ}\ar@{.}[uur]\ar[r]& {\circ} \\
} %5
}
}

By definition, we have the following complete $\tau$-slice $Q$ of $\olQ$.
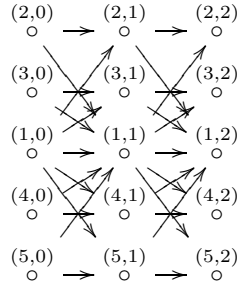
$$\xymatrix@C=0.4cm@R0.1cm{
\stackrel{(2,0)}{\circ} \ar[ddr]\ar[r]& \stackrel{(2,1)}{\circ} \ar[ddr]\ar[r] &\stackrel{(2,2)}{\circ}&&&  \\
\stackrel{(3,0)}{\circ} \ar[dr]\ar[r]& \stackrel{(3,1)}{\circ} \ar[dr]\ar[r] &\stackrel{(3,2)}{\circ}&&&  \\
\stackrel{(1,0)}{\circ}\ar[ddr] \ar[dr] \ar[ur]\ar[uur]\ar[r]& \stackrel{(1,1)}{\circ}\ar[ddr] \ar[dr] \ar[ur]\ar[uur]\ar[r]&\stackrel{(1,2)}{\circ}&&&\\
 \stackrel{(4,0)}{\circ} \ar[ur]\ar[r]&  \stackrel{(4,1)}{\circ} \ar[ur]\ar[r]& \stackrel{(4,2)}{\circ}&&&  \\
\stackrel{(5,0)}{\circ}\ar[uur]\ar[r]& \stackrel{(5,1)}{\circ}  \ar[uur]\ar[r] &\stackrel{(5,2)}{\circ}&&&  \\
}.$$
Let $\LL$ be the $\tau$-slice algebra and let $\GG$ be the dual $\tau$-slice algebra, that is, the quadratic dual of $\LL$.
Then by Theorem 6.1 of \cite{gw18}, $\GG$ is a quasi $2$-Fano algebra.
In $Q$, write $\alpha_{(i,t)}$ for the arrow from $(1,t)$ to $(i,t+1)$, $\beta_{(i,t)}$ for the arrow from $(i,t)$ to $(1,t+1)$ for $2\le i \le 5$, and $\gamma_{(i,t)}$ for the arrow from $(i,t)$ to $(i,t+1)$ for $2\le i \le 5$, $\GG$ is defined by the quiver $Q^{\perp}$ relations
$$\arr{rl}{\rho^{\perp} =& \{\sum_{i=1}^5\beta_{(i,1)}\alpha_{(i,0)}\} \cup \{\gamma_{(1,t+1)}\beta_{(i,t)}+\beta_{(i,t)}\gamma_{(i,t)}|2\le i\le 5, t=0,1\}\\&\cup\{\alpha_{(i,1)}\beta_{(i,0)}| 2\le i\le 5\} \cup \{\gamma_{(i,t+1)}\alpha_{(i,t)}+\alpha_{(i,t)}\gamma_{(1,t)}|2\le i\le 5, t=0,1\}.}$$
Using $\tau$-mutations on $Q$, we get all the  non-isomorphic complete $\tau$-slices.
The dual $\tau$-slice algebras of these complete $\tau$-slices are all the iterated $2$-APR tilting and cotilting algebras obtained from $\GG$.
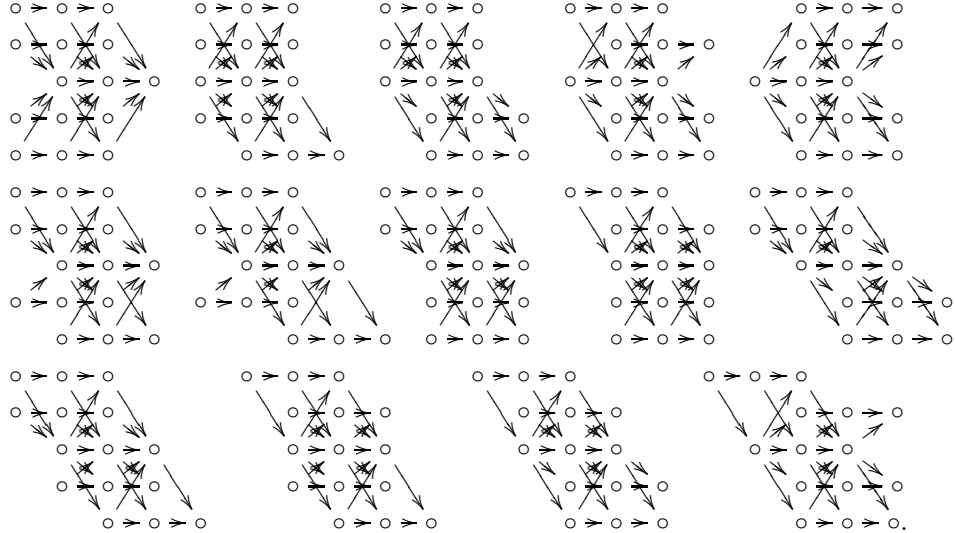
$$\xymatrix@C=0.2cm@R0.1cm{
{\circ} \ar[ddr]\ar[r]& {\circ} \ar[ddr]\ar[r] &{\circ}\ar[ddr]  &{}
  &{\circ} \ar[ddr]\ar[r]& {\circ} \ar[ddr]\ar[r] &{\circ} &{}
    &{\circ} \ar[ddr]\ar[r]& {\circ} \ar[ddr]\ar[r] &{\circ} & {}
      &{\circ} \ar[ddr]\ar[r]& {\circ} \ar[ddr]\ar[r] &{\circ} & {}
        && {\circ} \ar[ddr]\ar[r] &{\circ}\ar[r] & {\circ}{} \\ %1
{\circ} \ar[dr]\ar[r]& {\circ} \ar[dr]\ar[r] &{\circ} \ar[dr]  &{}
  &{\circ} \ar[dr]\ar[r]& {\circ} \ar[dr]\ar[r] &{\circ}  & {}
    &{\circ} \ar[dr]\ar[r]& {\circ} \ar[dr]\ar[r] &{\circ}  & {}
      && {\circ} \ar[dr]\ar[r] &{\circ}\ar[r] & {\circ}  {}
        && {\circ} \ar[dr]\ar[r] &{\circ}  \ar[r] &{\circ} {} \\ %2
& {\circ}\ar[ddr] \ar[dr] \ar[ur] \ar[uur] \ar[r] & {\circ} \ar[r] & {\circ}     {}
  &{\circ}\ar[ddr] \ar[dr] \ar[ur]\ar[uur]\ar[r]& {\circ}\ar[ddr] \ar[dr] \ar[ur]\ar[uur]\ar[r]&{\circ}  \ar[ddr]    & {}
    &{\circ}\ar[ddr] \ar[dr] \ar[ur]\ar[uur]\ar[r]& {\circ}\ar[ddr] \ar[dr] \ar[ur]\ar[uur]\ar[r]&{\circ} \ar[ddr] \ar[dr]      & {}
      &{\circ}\ar[ddr] \ar[dr] \ar[ur]\ar[uur]\ar[r]& {\circ}\ar[ddr] \ar[dr] \ar[ur] \ar[uur] \ar[r] & {\circ} \ar[ddr] \ar[dr] \ar[ur]    &   {}
        &{\circ}\ar[ddr] \ar[dr] \ar[ur]\ar[uur]\ar[r]& {\circ}\ar[ddr] \ar[dr] \ar[ur]\ar[uur]\ar[r]&{\circ} \ar[ddr] \ar[dr] \ar[ur]\ar[uur]     & {}\\ %3
{\circ} \ar[ur]\ar[r]& {\circ} \ar[ur]\ar[r]& {\circ}\ar[ur] & {}
  &{\circ} \ar[ur]\ar[r]& {\circ} \ar[ur]\ar[r]& {\circ} & {}
    && {\circ} \ar[ur]\ar[r]& {\circ} \ar[r]&{\circ} {}
      && {\circ} \ar[ur]\ar[r]& {\circ} \ar[r]&{\circ} {}
        && {\circ} \ar[ur]\ar[r]& {\circ}\ar[r]& {\circ}{}  \\ %4
{\circ}\ar[uur]\ar[r]&{\circ}  \ar[uur]\ar[r] &{\circ} \ar[uur] & {}
  &&{\circ}  \ar[uur]\ar[r] &{\circ}\ar[r]  &{\circ}{}
    &&{\circ}  \ar[uur]\ar[r] &{\circ}\ar[r] & {\circ} {}
      &&{\circ}  \ar[uur]\ar[r] &{\circ}\ar[r] & {\circ} {}
        &&{\circ}  \ar[uur]\ar[r] &{\circ}\ar[r] &{\circ} {} \\ %5
{\circ} \ar[ddr]\ar[r]& {\circ} \ar[ddr]\ar[r] &{\circ}\ar[ddr]  &{}
  &{\circ} \ar[ddr]\ar[r]& {\circ} \ar[ddr]\ar[r] &{\circ} \ar[ddr]  &{}
    &{\circ} \ar[ddr]\ar[r]& {\circ} \ar[ddr]\ar[r] &{\circ}\ar[ddr] &  {}
      &{\circ} \ar[ddr]\ar[r]& {\circ} \ar[ddr]\ar[r] &{\circ}\ar[ddr] & {}
        &{\circ} \ar[ddr]\ar[r] & {\circ} \ar[ddr]\ar[r] &{\circ}\ar[ddr] & {} \\%6
{\circ} \ar[dr]\ar[r]& {\circ} \ar[dr]\ar[r] &{\circ} \ar[dr]  &{}
  &{\circ} \ar[dr]\ar[r]& {\circ} \ar[dr]\ar[r] &{\circ} \ar[dr]&  {}
    &{\circ} \ar[dr]\ar[r]& {\circ} \ar[dr]\ar[r] &{\circ} \ar[dr] & {}
      && {\circ} \ar[dr]\ar[r] &{\circ}\ar[r]\ar[dr] &{\circ}   {}
        &{\circ} \ar[dr]\ar[r] & {\circ} \ar[dr]\ar[r] &{\circ}\ar[dr]&& {} \\%7
& {\circ}\ar[ddr] \ar[dr] \ar[ur] \ar[uur] \ar[r] & {\circ}\ar[ddr]\ar[r] & {\circ}     {}
  && {\circ}\ar[ddr] \ar[dr] \ar[ur]\ar[uur]\ar[r]&{\circ}  \ar[ddr]   \ar[r] & {\circ} \ar[ddr] &{}
    & {\circ}\ar[ddr] \ar[dr] \ar[ur]\ar[uur]\ar[r]&{\circ} \ar[ddr] \ar[dr] \ar[r] & {\circ}{}
     & & {\circ}\ar[ddr] \ar[dr] \ar[ur] \ar[uur] \ar[r] & {\circ} \ar[ddr] \ar[dr] \ar[ur]   \ar[r] &  {\circ} {}
         && {\circ}\ar[ddr] \ar[dr] \ar[ur]\ar[uur]\ar[r]&{\circ} \ar[ddr] \ar[dr] \ar[r]  &  {\circ} \ar[ddr] \ar[dr]    & {}\\ %8
{\circ} \ar[ur]\ar[r]& {\circ} \ar[ur]\ar[r]& {\circ}\ar[ur] & {}
  &{\circ} \ar[ur]\ar[r]& {\circ} \ar[ur]\ar[r]& {\circ}\ar[ur] & &{}
    & {\circ} \ar[ur]\ar[r]& {\circ} \ar[ur] \ar[r]& {\circ}{}
      && {\circ} \ar[ur]\ar[r]& {\circ} \ar[r] \ar[ur]& {\circ} {}
        &&& {\circ}\ar[r]\ar[ur]& {\circ}\ar[r]& {\circ} {}  \\ %9
&{\circ}  \ar[uur]\ar[r] &{\circ} \ar[uur]\ar[r] &{\circ} {}
  && &{\circ}\ar[r]\ar[uur] &{\circ}\ar[r] & {\circ}{}
    &{\circ}  \ar[uur]\ar[r] &{\circ}\ar[r]\ar[uur] & {\circ}{}
      &&{\circ}  \ar[uur]\ar[r] &{\circ}\ar[uur]\ar[r] & {\circ}{}
        && &{\circ}\ar[r]\ar[uur] &{\circ}  \ar[r] &{\circ} {} \\ %10
%%%%%%%%%%%%%%%%%%%%%%%%%%%%%%%%%%%%%%%%%%%%%%%%%%%%%%%%%%%%%%%%%%%
{\circ} \ar[ddr]\ar[r]& {\circ} \ar[ddr]\ar[r] &{\circ}\ar[ddr] & &{}
  &{\circ} \ar[ddr]\ar[r]& {\circ} \ar[ddr]\ar[r] &{\circ} \ar[ddr] & &{}
    &{\circ} \ar[ddr]\ar[r]& {\circ} \ar[ddr]\ar[r] &{\circ}\ar[ddr] & & {}
      &{\circ} \ar[ddr]\ar[r]& {\circ} \ar[ddr]\ar[r] &{\circ}\ar[ddr]& & {}
        \\ %11
{\circ} \ar[dr]\ar[r]& {\circ} \ar[dr]\ar[r] &{\circ} \ar[dr] & &{}
  && {\circ} \ar[dr]\ar[r] &{\circ} \ar[dr]\ar[r] & {\circ}& {}
    && {\circ} \ar[dr]\ar[r] &{\circ} \ar[dr]\ar[r] & {\circ} &{}
      &&  &{\circ}\ar[r]\ar[dr] &{\circ} \ar[r] & {\circ}  {}
        \\ %12
& {\circ}\ar[ddr] \ar[dr] \ar[ur] \ar[uur] \ar[r] & {\circ}\ar[ddr]\ar[dr]\ar[r] & {\circ}\ar[ddr]    & {}
  && {\circ}\ar[ddr] \ar[dr] \ar[ur]\ar[uur]\ar[r]&{\circ}  \ar[ddr] \ar[dr] \ar[ur]  \ar[r] & {\circ} \ar[ddr] &{}
    && {\circ}\ar[ddr] \ar[dr] \ar[ur]\ar[uur]\ar[r]&{\circ} \ar[ddr] \ar[ur] \ar[dr] \ar[r] & {\circ}\ar[ddr] \ar[dr] &{}
      && {\circ}\ar[ddr] \ar[dr] \ar[ur] \ar[uur] \ar[r] & {\circ} \ar[ddr] \ar[dr] \ar[ur]   \ar[r] &  {\circ}\ar[ddr] \ar[dr] \ar[ur]& {}
        \\ %13
& {\circ} \ar[ur]\ar[r]& {\circ}\ar[ur] \ar[r]&{\circ} &{}
  && {\circ} \ar[ur]\ar[r]& {\circ}\ar[ur]\ar[r] & {\circ}&{}
    && & {\circ} \ar[ur] \ar[r]& {\circ} \ar[r] & {\circ}{}
      && & {\circ} \ar[r] \ar[ur]& {\circ}\ar[r] & {\circ} {}
         \\ %14
& &{\circ} \ar[uur]\ar[r] &{\circ}\ar[r] &{\circ} {}
  && &{\circ}\ar[r]\ar[uur] &{\circ}\ar[r] & {\circ}{}
    && &{\circ}\ar[r]\ar[uur] & {\circ}\ar[r] & {\circ}{}
      && &{\circ}\ar[uur]\ar[r] & {\circ}\ar[r] & {\circ}{}
%%%%%%%%%%%%%%%%%%%%%%%%%%%%%%%%%%%%
.}$$
%}
\end{exa}

\section*{Acknowledgements}
We would like to thank  the referees  for reading the manuscript carefully and for suggestions and comments on revising and improving the paper.

{}

\end{document}